\documentclass{amsart}

\usepackage{latexsym,amsfonts,amssymb,exscale,enumerate,comment}
\usepackage{amsmath,amsthm,amscd}

\usepackage{url}
\usepackage[bookmarks=true,%
    colorlinks=true,%
    linkcolor=blue,%
    citecolor=blue,%
    filecolor=blue,%
    menucolor=blue,%
    urlcolor=blue,%
    breaklinks=true]{hyperref}

\input xy
\usepackage[all]{xy}
\xyoption{line}
\xyoption{arrow}
\xyoption{color}
\SelectTips{cm}{}
\usepackage{tikz}
\usetikzlibrary{shapes,snakes}
\usetikzlibrary{decorations.markings}
\usetikzlibrary{decorations.pathreplacing}
\tikzstyle directed=[postaction={decorate,decoration={markings,
    mark=at position #1 with {\arrow{>}}}}]

\newcommand{\hackcenter}[1]{
 \xy (0,0)*{#1}; \endxy}

\tikzset{->-/.style={decoration={
  markings,
  mark=at position #1 with {\arrow{>}}},postaction={decorate}}}

\tikzset{middlearrow/.style={
        decoration={markings,
            mark= at position 0.5 with {\arrow{#1}} ,
        },
        postaction={decorate}
    }
}

\usepackage{bm}

\def\u0{{\underline{0}}}

\def\la{\langle}
\def\ra{\rangle}
\newcommand{\sE}{\cal{E}}
\newcommand{\sF}{\cal{F}}

\newcommand{\onel}{\1_{\lambda}}

\def\Id{\mathrm{Id}}

\theoremstyle{plain}
\newtheorem{theorem}{Theorem}

\newtheorem{proposition}[theorem]{Proposition}

\theoremstyle{definition}

\newtheorem{definition}[theorem]{Definition}

\theoremstyle{definition}
\newtheorem{remark}[theorem]{Remark}
\numberwithin{equation}{section}
\numberwithin{theorem}{section}
\newcommand{\maps}{\colon}

\newcommand{\refequal}[1]{\xy {\ar@{=}^{#1}
(-1,0)*{};(1,0)*{}};
\endxy}
\newcommand{\Hom}{{\rm Hom}}

\renewcommand{\to}{\rightarrow}

\def\Id{\mathrm{Id}}

\def\mf{\mathfrak}

\numberwithin{equation}{section}
\let\hat=\widehat

\let\epsilon=\varepsilon

\usepackage{bbm}

\def\Z{{\mathbbm Z}}
\def\cal#1{\mathcal{#1}}%
\def\1{\mathbbm{1}}%
\def\nn{\notag}

\def\la{\langle}
\def\ra{\rangle}

\renewcommand{\l}{\lambda}

\def\cal#1{\mathcal{#1}}

\newcommand\nc{\newcommand}
\nc\rnc{\renewcommand}
\nc\Kar{\operatorname{Kar}}
\nc\End{\operatorname{End}}

\nc\sfc{\textsf{c}}

\newcommand{\scs}{\scriptstyle}
\newcommand{\Ucat}{\cal{U}}

\nc\Sym{\operatorname{Sym}}

\allowdisplaybreaks

\title{Parameters in categorified quantum groups}

\begin{document}
\setcounter{tocdepth}{3}

%\author{Sabin Cautis}
%\email{cautis@math.ubc.ca}
%\address{Department of Mathematics \\ University of British Columbia \\ Vancouver, BC}
%
%
\author{Aaron D. Lauda}
\email{lauda@usc.edu}
\address{Department of Mathematics\\ University of Southern California \\ Los Angeles, CA}

%%
%\author{Joshua Sussan}
%\email{jsussan@mec.cuny.edu}
%\address{Department of Mathematics \\ CUNY Medgar Evers \\ Brooklyn, NY}
%\date{\today}

%
\maketitle

\begin{abstract}
In this note we give explicit isomorphisms of 2-categories between various versions of the categorified quantum group associated to a simply-laced Kac-Moody algebra.  These isomorphisms are convenient when working with the categorified quantum group.
They make it possible to translate results from the $\mf{gl}_n$ variant of the 2-category to the $\mf{sl}_n$ variant and transfer results between various conventions in the literature.  We also extend isomorphisms of finite type KLR algebras for different choices of parameters to the level of 2-categories.
\end{abstract}
\setcounter{tocdepth}{2}
%\tableofcontents

\section{Introduction}

Brundan has shown that the categorified quantum group associated to $\mf{g}$ is essentially unique~\cite{Brundan2}.  There remain several choices to be made:
\begin{itemize}
  \item a choice of scalars $Q$ determining the KLR algebra $R_Q$ categorifying the positive half of $\dot{U}(\mf{g})$,
  \item overall scalings of generating 2-morphisms that affect the precise form of the relations.  These parameters include the values of degree zero bubbles  and the dependence of these choices on the weight $\l \in X$.
\end{itemize}
Different choices for the parameters can result in different behaviors of the resulting 2-categories.

Skew Howe duality has proven to be a powerful tool in higher representation theory and its applications to link homology~\cite{CKL-skew,Cautis,LQR,QR,MW}.  However, this technique is fundamentally a $\mf{gl}_n$ phenomenon, though many of the results in the literature express the results in terms of $\mf{sl}_n$.
Hence, one has to move between the categorified quantum groups for $\mf{gl}_n$ and $\mf{sl}_n$. This creates complications because in the current literature the categorifications of $\mf{sl}_n$ and $\mf{gl}_n$ are not connected in a straightforward manner.  The parameters in each case seem oddly related.  Further, many results  are formulated for some fixed choice of parameters in the categorified quantum group.  For example, in type $A$ the connection with Soergel bimodules relies on one choice of coefficients~\cite{MSV}, while the connection to cohomology rings of Grassmannians utilizes another~\cite{Lau1,Lau2,KL3}.
Translating results from one choice to another can be rather complicated since one must paste together and extend the needed rescaling 2-functors from partial results scattered throughout the literature, and some of these isomorphisms require weight dependent rescalings that depend on values of weights modulo 4.

%Skew Howe and annoying signs converting from $\mf{gl}_n$ to $\mf{sl}_n$} \cite{LQR} \cite[Section 5.3.3]{MPT}

%Bubble parameters. Inverse vs non inverse. Constant along sl2 strings or not.  sl(2) relations $\pm 1$.

In this note we define a general version of the categorified quantum group and give explicit 2-functors for translating results from one formulation to the other.  This allows for an immediate translation from $\mf{gl}_n$ coefficients to $\mf{sl}_n$ coefficients with an explicit rescaling 2-functor collecting various results from the literature. We also make explicit certain  isomorphisms that can be defined between KLR-algebras $R_Q$ and $R_{Q'}$ for different choices of scalars when the underlying graph of the simply-laced Kac-Moody algebra is a tree, in particular, a Dynkin diagram.  Here we show how to extend these isomorphisms to isomorphisms of 2-categories for the corresponding categorified quantum groups.  This map hasn't appeared in the literature before.   The author  found himself often reconstructing these results and decided to produce this note.

\subsection{Versions of categorified quantum groups}

In \cite{KL3} a 2-category $\cal{U}(\mf{g})$ was defined for any symmetrizable Kac-Moody algebra $\mf{g}$.  They conjectured that this 2-category
categorifies the integral form of Lusztig's idempotented version $\dot{U}_q(\mf{g})$ of the quantum group and proved the conjecture for $\mf{g}=\mf{sl}_n$, following earlier work for $\mf{g}=\mf{sl}_2$ from \cite{Lau1}.  Later Webster proved the conjecture in general~\cite{Web}.   An important feature of the 2-category $\cal{U}(\mf{g})$ is its diagrammatic nature, making it possible to represent 2-morphisms using a planar diagrammatic calculus.  However, this 2-category is only cyclic up to a sign, meaning that planar deformations of a diagram representing a 2-morphism only gives rise to scalar multiples of the same 2-morphism.

The KLR algebra governs the upward oriented strands in $\cal{U}(\mf{g})$.
In \cite{KL2,Rou2} more general coefficients $Q$ for KLR algebras were introduced and a more general non-cyclic 2-category $\cal{U}_Q(\mf{g})$ was defined in \cite{CLau} for an arbitrary KLR algebra $R_Q$.  Closely related categories were also studied in~\cite{Rou2}.
Composing biadjunction units and counits with other generators produces certain endomorphisms of the identity map in each weight $\l$ called \emph{dotted bubbles}.  With the appropriate number of dots, these dotted bubbles have degree zero and take some value called \emph{bubble parameters} in the ground field $\Bbbk^{\times}$:
\begin{equation} \label{eq:degreezero}
 \hackcenter{ \begin{tikzpicture} [scale=.8]
 \draw (-.15,.35) node { $\scs i$};
 \draw  (0,0) arc (180:360:0.5cm) [thick];
 \draw[,<-](1,0) arc (0:180:0.5cm) [thick];
\filldraw  [black] (.1,-.25) circle (2.5pt);
 \node at (-.5,-.5) {\tiny $\l_i -1$};
 \node at (1.15,1) { $\lambda  $};
\end{tikzpicture} }
\;\; =
\sfc_{i,\l}^+, \qquad \quad
\;
\hackcenter{ \begin{tikzpicture} [scale=.8]
 \draw (-.15,.35) node { $\scs i$};
 \draw  (0,0) arc (180:360:0.5cm) [thick];
 \draw[->](1,0) arc (0:180:0.5cm) [thick];
\filldraw  [black] (.9,-.25) circle (2.5pt);
 \node at (1.35,-.5) {\tiny $-\l_i -1$};
 \node at (1.15,1) { $\lambda $};
\end{tikzpicture} }
\;\; =
\sfc_{i,\l}^{-} .
\end{equation}
In the 2-category $\cal{U}_Q(\mf{g})$ all these bubble parameters are normalized to $1$.

A categorification $\cal{U}(\mf{gl}_n)$ of quantum $\mf{gl}_n$ and its Schur quotients were defined in \cite{MSV}.  Unlike the 2-category $\cal{U}_Q(\mf{g})$, this 2-category is cyclic.  However, the bubble parameters in $\cal{U}(\mf{gl}_n)$ take values of $\pm 1$ depending explicitly on the $\mf{gl}(n)$ weights.  Further, the relations that give isomorphisms lifting the $\mf{sl}_2$ relations
have a significant sign difference from $\cal{U}_Q(\mf{g})$ that force the values of degree zero bubbles to alternate sign as the weight changes along $\mf{sl}_2$ strings.
Additionally, the 2-category $\cal{U}(\mf{gl}_n)$ has a specific choice of scalars $Q$ for the KLR algebra that depends on an orientation of the underlying quiver.

In  \cite[Section 3.9]{Lau3} a version of the $\mf{sl}_2$ 2-category $\cal{U}_{\beta}(\mf{sl}_2)$ was defined with arbitrary values of bubbles parameters and a more general $\mf{sl}_2$ relation
\begin{equation}
\begin{split}
\hackcenter{\begin{tikzpicture}[scale=0.8]
    \draw[thick, ->] (0,0) -- (0,2);
    \draw[thick, <-] (.75,0) -- (.75,2);
     \node at (-.2,.2) {\tiny $i$};
    \node at (.95,.2) {\tiny $i$};
\end{tikzpicture}}
\;\; = \;\; \textcolor[rgb]{1.00,0.00,0.00}{\beta_{i,\l}} \;
 \hackcenter{\begin{tikzpicture}[scale=0.8]
    \draw[thick] (0,0) .. controls ++(0,.5) and ++(0,-.5) .. (.75,1);
    \draw[thick,<-] (.75,0) .. controls ++(0,.5) and ++(0,-.5) .. (0,1);
    \draw[thick] (0,1 ) .. controls ++(0,.5) and ++(0,-.5) .. (.75,2);
    \draw[thick, ->] (.75,1) .. controls ++(0,.5) and ++(0,-.5) .. (0,2);
        \node at (-.2,.15) {\tiny $i$};
    \node at (.95,.15) {\tiny $i$};
\end{tikzpicture}}
\;\; - \;\; \textcolor[rgb]{1.00,0.00,0.00}{\beta_{i,\l}} \;
\sum_{\overset{f_1+f_2+f_3}{=\l_i-1}}\hackcenter{
 \begin{tikzpicture}[scale=0.8]
 \draw[thick,->] (0,-1.0) .. controls ++(0,.5) and ++ (0,.5) .. (.8,-1.0) node[pos=.75, shape=coordinate](DOT1){};
  \draw[thick,<-] (0,1.0) .. controls ++(0,-.5) and ++ (0,-.5) .. (.8,1.0) node[pos=.75, shape=coordinate](DOT3){};
 \draw[thick,->] (0,0) .. controls ++(0,-.45) and ++ (0,-.45) .. (.8,0)node[pos=.25, shape=coordinate](DOT2){};
 \draw[thick] (0,0) .. controls ++(0,.45) and ++ (0,.45) .. (.8,0);
 \draw (-.15,.7) node { $\scs i$};
\draw (1.05,0) node { $\scs i$};
\draw (-.15,-.7) node { $\scs i$};
% \draw (.3,.125) node {};
% \draw  (0,0) arc (180:360:0.5cm) [thick];
% \draw[,<-](1,0) arc (0:180:0.5cm) [thick];
%\filldraw  [black] (.1,-.25) circle (2.5pt);
 \node at (.95,.65) {\tiny $f_3$};
 \node at (-.55,-.05) {\tiny $\overset{-\l_i-1}{+f_2}$};
  \node at (.95,-.65) {\tiny $f_1$};
 \node at (1.95,.3) { $\lambda $};
 \filldraw[thick]  (DOT3) circle (2.5pt);
  \filldraw[thick]  (DOT2) circle (2.5pt);
  \filldraw[thick]  (DOT1) circle (2.5pt);
\end{tikzpicture} }
\\
\hackcenter{\begin{tikzpicture}[scale=0.8]
    \draw[thick, <-] (0,0) -- (0,2);
    \draw[thick, ->] (.75,0) -- (.75,2);
     \node at (-.2,.2) {\tiny $i$};
    \node at (.95,.2) {\tiny $i$};
\end{tikzpicture}}
\;\; = \;\; \textcolor[rgb]{1.00,0.00,0.00}{\beta_{i,\l}} \;
 \hackcenter{\begin{tikzpicture}[scale=0.8]
    \draw[thick,<-] (0,0) .. controls ++(0,.5) and ++(0,-.5) .. (.75,1);
    \draw[thick] (.75,0) .. controls ++(0,.5) and ++(0,-.5) .. (0,1);
    \draw[thick, ->] (0,1 ) .. controls ++(0,.5) and ++(0,-.5) .. (.75,2);
    \draw[thick] (.75,1) .. controls ++(0,.5) and ++(0,-.5) .. (0,2);
        \node at (-.2,.15) {\tiny $i$};
    \node at (.95,.15) {\tiny $i$};
\end{tikzpicture}}
\;\; - \;\; \textcolor[rgb]{1.00,0.00,0.00}{\beta_{i,\l}} \;
\sum_{\overset{f_1+f_2+f_3}{=-\l_i-1}}\hackcenter{
 \begin{tikzpicture}[scale=0.8]
 \draw[thick,<-] (0,-1.0) .. controls ++(0,.5) and ++ (0,.5) .. (.8,-1.0) node[pos=.75, shape=coordinate](DOT1){};
  \draw[thick,->] (0,1.0) .. controls ++(0,-.5) and ++ (0,-.5) .. (.8,1.0) node[pos=.75, shape=coordinate](DOT3){};
 \draw[thick ] (0,0) .. controls ++(0,-.45) and ++ (0,-.45) .. (.8,0)node[pos=.25, shape=coordinate](DOT2){};
 \draw[thick, ->] (0,0) .. controls ++(0,.45) and ++ (0,.45) .. (.8,0);
 \draw (-.15,.7) node { $\scs i$};
\draw (1.05,0) node { $\scs i$};
\draw (-.15,-.7) node { $\scs i$};
% \draw (.3,.125) node {};
% \draw  (0,0) arc (180:360:0.5cm) [thick];
% \draw[,<-](1,0) arc (0:180:0.5cm) [thick];
%\filldraw  [black] (.1,-.25) circle (2.5pt);
 \node at (.95,.65) {\tiny $f_3$};
 \node at (-.55,-.05) {\tiny $\overset{\l_i-1}{+f_2}$};
  \node at (.95,-.65) {\tiny $f_1$};
 \node at (1.95,.3) { $\lambda $};
 \filldraw[thick]  (DOT3) circle (2.5pt);
  \filldraw[thick]  (DOT2) circle (2.5pt);
  \filldraw[thick]  (DOT1) circle (2.5pt);
\end{tikzpicture} } \label{eq:sl2}
\end{split}
\end{equation}
that generalizes the corresponding relation from the $\cal{U}(\mf{gl}_n)$ when $\beta_{i,\l}=+1$ and $\cal{U}_Q(\mf{g})$ when $\beta_{i,\l}=-1$ for all $i\in I$ and $\l \in X$.
For any choice of parameters $\beta_{i,\l}$, the 2-category $\cal{U}_{\beta}(\mf{sl}_2)$ is isomorphic to the original 2-category $\cal{U}(\mf{sl}_2)$ from \cite{Lau1}, although this isomorphism requires rescalings of caps and cups that depend on the weight modulo 4.  These parameters $\beta_{i,\l}$ control the values of higher degree bubbles and the definition of the `fake bubbles' as defined by the infinite Grassmannian equation
\begin{align}
  \left(
  \hackcenter{ \begin{tikzpicture} [scale=.6]
    % Bubble
 \draw  (-.75,1) arc (360:180:.45cm) [thick];
 \draw[<-](-.75,1) arc (0:180:.45cm) [thick];
     \filldraw  [black] (-1.55,.75) circle (2.5pt);
        \node at (-1.3,.3) { $\scriptstyle \l_i -1$};
        \node at (-1.4,1.7) { $i $};
 %% L
 \node at (-.2,1.5) { $\lambda $};
\end{tikzpicture}} \; t^0 \; + \;
  \hackcenter{ \begin{tikzpicture} [scale=.6]
    % Bubble
 \draw  (-.75,1) arc (360:180:.45cm) [thick];
 \draw[<-](-.75,1) arc (0:180:.45cm) [thick];
     \filldraw  [black] (-1.55,.75) circle (2.5pt);
        \node at (-1.3,.3) { $\scriptstyle \l_i -1+1$};
        \node at (-1.4,1.7) { $i $};
 %% L
 \node at (-.2,1.5) { $\lambda $};
\end{tikzpicture}} \; t^1
\;\; + \;\; \dots \;\; + \;\;
  \hackcenter{ \begin{tikzpicture} [scale=.6]
    % Bubble
 \draw  (-.75,1) arc (360:180:.45cm) [thick];
 \draw[<-](-.75,1) arc (0:180:.45cm) [thick];
     \filldraw  [black] (-1.55,.75) circle (2.5pt);
        \node at (-1.3,.3) { $\scriptstyle \l_i -1 + r$};
        \node at (-1.4,1.7) { $i $};
 %% L
 \node at (-.2,1.5) { $\lambda $};
\end{tikzpicture}} \; t^r
\;\; + \;\; \dots
  \right) \times  \hspace{1in} \nn \\
  \left(
  \hackcenter{ \begin{tikzpicture} [scale=.6]
    % Bubble
 \draw  (-.75,1) arc (360:180:.45cm) [thick];
 \draw[->](-.75,1) arc (0:180:.45cm) [thick];
     \filldraw  [black] (-1.55,.75) circle (2.5pt);
        \node at (-1.3,.3) { $\scriptstyle -\l_i -1$};
        \node at (-1.4,1.7) { $i $};
 %% L
 \node at (-.2,1.5) { $\lambda $};
\end{tikzpicture}} \; t^0 \; + \;
\hackcenter{ \begin{tikzpicture} [scale=.6]
    % Bubble
 \draw  (-.75,1) arc (360:180:.45cm) [thick];
 \draw[->](-.75,1) arc (0:180:.45cm) [thick];
     \filldraw  [black] (-1.55,.75) circle (2.5pt);
        \node at (-1.3,.3) { $\scriptstyle -\l_i -1 + 1$};
        \node at (-1.4,1.7) { $i $};
 %% L
 \node at (-.2,1.5) { $\lambda $};
\end{tikzpicture}}\; t
\; + \; \dots \; + \;
\hackcenter{ \begin{tikzpicture} [scale=.6]
    % Bubble
 \draw  (-.75,1) arc (360:180:.45cm) [thick];
 \draw[->](-.75,1) arc (0:180:.45cm) [thick];
     \filldraw  [black] (-1.55,.75) circle (2.5pt);
        \node at (-1.3,.3) { $\scriptstyle -\l_i -1 + s$};
        \node at (-1.4,1.7) { $i $};
 %% L
 \node at (-.2,1.5) { $\lambda $};
\end{tikzpicture}}\; t^s
\; + \; \dots
  \right)
  = \textcolor[rgb]{1.00,0.00,0.00}{-\frac{1}{\beta_{i,\l}}  }.
\end{align}
The parameters $\beta_{i,\l}$ also impact how bubble parameters change as weights change along $\mf{sl}_2$ strings.

In \cite{BHLW2} a cyclic variant $\cal{U}_Q^{cyc}(\mf{g})$ of the categorified quantum group was defined utilizing more general bubble coefficients.
This 2-category has $\mf{sl}_2$ relations with all $\beta_{i,\l}=-1$ so that bubble coefficients $\sfc_{i,\l}^+$ are constant along $\mf{sl}_2$ strings and $\sfc_{i,\l}^- = \left( \sfc_{i,\l}^+\right)^{-1}$.
In \cite[Theorem 2.1]{BHLW2} an explicit isomorphism of 2-categories is defined from $\cal{U}_Q^{cyc}(\mf{g})$ to $\cal{U}_Q(\mf{g})$.
A minimal presentation of this category can be determined from \cite{Brundan2}.

Here we introduce a variant $\cal{U}_{Q,\beta}(\mf{g})=\cal{U}_{Q,\beta}^{cyc}(\mf{g})$ of the categorified quantum group (Definition~\ref{def:Ubeta}) that generalizes all of the variants discussed above.  We then give a 2-isomorphism (Theorem~\ref{thm:beta}) showing this general 2-category is isomorphic to the cyclic 2-category $\cal{U}_Q^{cyc}(\mf{g})$ and hence $\cal{U}_Q(\mf{g})$.  Again, this is entirely expected by Brundan's uniqueness of categorification result.   This 2-isomorphism makes it possible to directly translate results from the $\mf{gl}_n$ version of the categorified quantum group to the $\mf{sl}_n$ variant.

In Sections~\ref{sec:nilhecke} and \ref{sec:generalQ} we provide new results showing how to extend various isomorphisms of KLR-algebras to the level of 2-categories.

\subsection{Acknowledgements}
The author is partially supported by the NSF grants DMS-1255334 and DMS-1664240.  He would also like to thank Joshua Sussan and Hoel Queffelec for comments on an early version of this note.

% = = = ============================================================
%
\section{The categorified quantum group $\cal{U}_{Q,\beta}(\mf{g})$}
% = = = ============================================================

Below we indicate in \textcolor[rgb]{1.00,0.00,0.00}{red} the changes in the definition of the categorified quantum group, so that experts can easily see the differences.

% - - - - - - - - - - - - - - - - - - - - - - - - - - - - - - - - -
%
\subsubsection{Cartan data}
%
% - - - - - - - - - - - - - - - - - - - - - - - - - - - - - - - - -

For this article we restrict our attention to simply-laced Kac-Moody algebras. These algebras are associated to a symmetric Cartan data consisting of
\begin{itemize}
\item a free $\Z$-module $X$ (the weight lattice),
\item for $i \in I$ ($I$ is an indexing set) there are elements $\alpha_i \in X$ (simple roots) and $\Lambda_i \in X$ (fundamental weights),
\item for $i \in I$ an element $h_i \in X^\vee = \Hom_{\Z}(X,\Z)$ (simple coroots),
\item a bilinear form $(\cdot,\cdot )$ on $X$.
\end{itemize}
Write $\langle \cdot, \cdot \rangle \maps X^{\vee} \times X
\to \Z$ for the canonical pairing. This data should satisfy:
\begin{itemize}
\item $(\alpha_i, \alpha_i) = 2$ for any $i\in I$,
\item $(\alpha_i,\alpha_j) \in \{ 0, -1\}$  for $i,j\in I$ with $i \neq j$,
\item $\la i,\lambda\ra :=\langle h_i, \lambda \rangle =  (\alpha_i,\lambda)$
  for $i \in I$ and $\lambda \in X$,
\item $\langle h_j, \Lambda_i \rangle =\delta_{ij}$ for all $i,j \in I$.
\end{itemize}
Hence $(a_{ij})_{i,j\in I}$ is a symmetrizable generalized Cartan matrix, where $a_{ij}=\langle
h_i, \alpha_j \rangle=(\alpha_i, \alpha_j)$.  We will sometimes denote the bilinear pairing $(\alpha_i,\alpha_j)$ by $i \cdot j$ and abbreviate $\la i,\lambda\ra$ to $\lambda_i$.
We denote by $X^+ \subset X$ the dominant weights which are of the form $\sum_i \lambda_i \Lambda_i$ where $\lambda_i \ge 0$.

%We write $W=W_{\mf{g}}$ for the Weyl group of type $\mf{g}$ and $B_{\mf{g}}$ for the corresponding braid group.
%Associated to a symmetric Cartan data is a graph $\Gamma$ without loops or multiple edges.  The vertices of $\Gamma$ are the elements of the set $I$ and there is an edge from vertex $i$ to vertex $j$ if and only if $(\alpha_i, \alpha_j) =-1$.

% - - - - - - - - - - - - -  - - - - - - - - - - - - -
\subsubsection{Parameters}
% - - - - - - - - - - - - -  - - - - - - - - - - - - -

We introduce new parameters following \cite{Lau3} that extend the bubble parameters introduced for the cyclic version of the quantum group in \cite{BHLW2}.
\begin{definition}
Associated to a Cartan datum we define {\em bubble parameters $\beta$} to be a set consisting of
\begin{itemize}
   \item $\beta_i = \beta_{i,\l}  \in \Bbbk^{\times}$ for $i\in I$ and $\lambda \in X$,
  \item $
%   \hackcenter{ \begin{tikzpicture} [scale=.8]
% \draw (-.15,.35) node { $\scs i$};
% \draw  (0,0) arc (180:360:0.5cm) [thick];
% \draw[,<-](1,0) arc (0:180:0.5cm) [thick];
%\filldraw  [black] (.1,-.25) circle (2.5pt);
% \node at (-.5,-.5) {\tiny $\l_i -1$};
% \node at (1.15,1) { $\lambda  $};
%\end{tikzpicture} }
%\;\; =
\sfc_{i,\l}^{+} \in \Bbbk^{\times}$  for $i\in I$ and $\lambda \in X$,
%with $\la i, \l\ra \geq 0$,
%
   \item $
%   \hackcenter{ \begin{tikzpicture} [scale=.8]
% \draw (-.15,.35) node { $\scs i$};
% \draw  (0,0) arc (180:360:0.5cm) [thick];
% \draw[->](1,0) arc (0:180:0.5cm) [thick];
%\filldraw  [black] (.9,-.25) circle (2.5pt);
% \node at (1.35,-.5) {\tiny $-\l_i -1$};
% \node at (1.15,1) { $\lambda $};
%\end{tikzpicture} }
%\;\; =
 \sfc_{i,\l}^- \in \Bbbk^{\times}$  for $i\in I$ and $\lambda \in X$.
\end{itemize}
Note that we do not require that $\sfc_{i,\l}^- = (\sfc_{i,\l}^+)^{-1}$ as was done in \cite{BHLW2}.

\begin{definition} \label{eq:Q}
Associated to a symmetric Cartan datum define a {\em choice of scalars $Q$} consisting of:
\begin{itemize}
  \item $\left\{ t_{ij}  \mid \text{ for all $i,j \in I$ with $i \neq j$} \right\}$,
\end{itemize}
such that
\begin{itemize}
\item  $t_{ij} \in \Bbbk^{\times}$ ,
 \item $t_{ij}=t_{ji}$ when $a_{ij}=0$.
\end{itemize}
\end{definition}
It is convenient to define $t_{ii}:= -\beta_i=-\beta_{i,\l}$ to express the compatibility condition \eqref{eq:c-change-sl} in a uniform manner.
A choice of bubble parameters is said to be {\em compatible with the scalars $Q$} if
\begin{equation} \label{eq:ccinv}
 \textcolor[rgb]{1.00,0.00,0.00}{ \sfc_{i,\lambda}^+ \sfc_{i,\lambda}^-= - \frac{1}{\beta_{i}} = \frac{1}{t_{ii}},}
\end{equation}
\begin{equation} \label{eq:c-change-sl}
  \textcolor[rgb]{1.00,0.00,0.00}{ \sfc_{i,\lambda\pm \alpha_j}^{\pm} =   t_{ij}\sfc_{i,\lambda}^{\pm}}.
\end{equation}
\end{definition}
Such a compatible choice of scalars can be chosen for any $t_{ij}$ by
fixing an arbitrary choice of $\sfc_{i,\lambda}^+$ for a fixed coset representative in every coset of the root lattice in the weight lattice, and then
extending to the rest of the coset using the compatibility conditions.

For any choice of bubble parameters compatible with the choice of scalars $Q$ the values along an $\mf{sl}_2$-string depend on $\beta_{i,\l}$ since
\[
 \sfc_{i,\l+\alpha_i} = t_{ii} \sfc_{i,\l} = \textcolor[rgb]{1.00,0.00,0.00}{(-\beta_i)}\sfc_{i,\l},
\]
so that for all $k \in \Z$ we have $\sfc_{i,\l+k\alpha_i}=\textcolor[rgb]{1.00,0.00,0.00}{(-\beta_{i,\l})^{k}}\sfc_{i,\l}$.

% - - - - - - - - - - - - -  - - - - - - - - - - - - -
\subsection{The general cyclic form of the categorified quantum group}
\label{2categorysection}
% - - - - - - - - - - - - -  - - - - - - - - - - - - -

\begin{definition} \label{def:Ubeta}
Let $\beta$ be a choice of bubble parameters that is compatible with a choice of scalars $Q$.
The 2-category  $\Ucat_{Q,\beta}(\mf{g}):= \Ucat_{Q,\beta}^{cyc}(\mf{g})$ is the graded linear 2-category consisting of:
\begin{itemize}
\item \textbf{Objects} $\lambda$ for $\lambda \in X$.
\item \textbf{1-morphisms} are formal direct sums of (shifts of) compositions of
$$\onel, \quad \1_{\lambda+\alpha_i} \sE_i= \1_{\lambda+\alpha_i} \sE_i\onel, \quad \text{ and }\quad
\1_{\lambda-\alpha_i} \sF_i= \1_{\lambda-\alpha_i} \sF_i\onel$$
for $i \in I$ and $\lambda \in X$.  We denote the grading shift by $\la 1 \ra$, so that for each 1-morphism $x$ in $\cal{U}$ and $t\in \Z$ we a 1-morphism $x\la t\ra$.

\item \textbf{2-morphisms} are $\Bbbk$-vector spaces spanned by compositions of coloured, decorated tangle-like diagrams illustrated below.
\begin{align}
\hackcenter{\begin{tikzpicture}[scale=0.8]
    \draw[thick, ->] (0,0) -- (0,1.5)
        node[pos=.5, shape=coordinate](DOT){};
    \filldraw  (DOT) circle (2.5pt);
    \node at (-.85,.85) {\tiny $\lambda +\alpha_i$};
    \node at (.5,.85) {\tiny $\lambda$};
    \node at (-.2,.1) {\tiny $i$};
\end{tikzpicture}} &\maps \cal{E}_i\onel \to \cal{E}_i\onel \la i\cdot i \ra  & \quad
 &
%\hackcenter{\begin{tikzpicture}[scale=0.8]
%    \draw[thick, <-] (0,0) -- (0,1.5)
%        node[pos=.5, shape=coordinate](DOT){};
%    \filldraw  (DOT) circle (2.5pt);
%    \node at (-.85,.85) {\tiny $\lambda -\alpha_i$};
%    \node at (.5,.85) {\tiny $\lambda$};
%    \node at (-.2,1.4) {\tiny $i$};
%\end{tikzpicture}}
%\maps \cal{F}_i\onel \to \cal{F}_i\onel\la i\cdot i \ra  \nn \\
%   & & & \nn \\
  \hackcenter{\begin{tikzpicture}[scale=0.8]
    \draw[thick, ->] (0,0) .. controls (0,.5) and (.75,.5) .. (.75,1.0);
    \draw[thick, ->] (.75,0) .. controls (.75,.5) and (0,.5) .. (0,1.0);
    \node at (1.1,.55) {\tiny $\lambda$};
    \node at (-.2,.1) {\tiny $i$};
    \node at (.95,.1) {\tiny $j$};
\end{tikzpicture}} \;\;\maps \cal{E}_i\cal{E}_j\onel  \to \cal{E}_j\cal{E}_i\onel\la -i\cdot j \ra
% &
%  &
%  \hackcenter{\begin{tikzpicture}[scale=0.8]
%    \draw[thick, <-] (0,0) .. controls (0,.5) and (.75,.5) .. (.75,1.0);
%    \draw[thick, <-] (.75,0) .. controls (.75,.5) and (0,.5) .. (0,1.0);
%    \node at (1.1,.55) {\tiny $\lambda$};
%    \node at (-.25,.1) {\tiny $i$};
%    \node at (1,.1) {\tiny $j$};
%\end{tikzpicture}}\;\; \maps \cal{F}_i\cal{F}_j\onel  \to \cal{F}_j\cal{F}_i\onel\la -i\cdot j \ra
 \nn \smallskip\\
\hackcenter{\begin{tikzpicture}[scale=0.8]
    \draw[thick, <-] (.75,2) .. controls ++(0,-.75) and ++(0,-.75) .. (0,2);
    \node at (.4,1.2) {\tiny $\lambda$};
    \node at (-.2,1.9) {\tiny $i$};
\end{tikzpicture}} \;\; &\maps \onel  \to \cal{F}_i\cal{E}_i\onel\la  1 + \lambda_i  \ra   &
    &
\hackcenter{\begin{tikzpicture}[scale=0.8]
    \draw[thick, ->] (.75,2) .. controls ++(0,-.75) and ++(0,-.75) .. (0,2);
    \node at (.4,1.2) {\tiny $\lambda$};
    \node at (.95,1.9) {\tiny $i$};
\end{tikzpicture}} \;\; \maps \onel  \to\cal{E}_i\cal{F}_i\onel\la  1 -  \lambda_i  \ra   \nn \smallskip \\
     % & & & \nn \\
\hackcenter{\begin{tikzpicture}[scale=0.8]
    \draw[thick, ->] (.75,-2) .. controls ++(0,.75) and ++(0,.75) .. (0,-2);
    \node at (.4,-1.2) {\tiny $\lambda$};
    \node at (.95,-1.9) {\tiny $i$};
\end{tikzpicture}} \;\; & \maps \cal{F}_i\cal{E}_i\onel \to\onel\la  1 +  \lambda_i  \ra   &
    &
\hackcenter{\begin{tikzpicture}[scale=0.8]
    \draw[thick, <-] (.75,-2) .. controls ++(0,.75) and ++(0,.75) .. (0,-2);
    \node at (.4,-1.2) {\tiny $\lambda$};
    \node at (-.2,-1.9) {\tiny $i$};
\end{tikzpicture}} \;\;\maps\cal{E}_i\cal{F}_i\onel  \to\onel\la  1 -  \lambda_i  \ra  \nn
\end{align}
\end{itemize}
In this $2$-category (and those throughout the paper) we
read diagrams from right to left and bottom to top.  The identity 2-morphism of the 1-morphism
$\cal{E}_i \onel$ is
represented by an upward oriented line labelled by $i$ and the identity 2-morphism of $\cal{F}_i \onel$ is
represented by a downward such line.

The 2-morphisms satisfy the following relations:
\begin{enumerate}
\item \label{item_cycbiadjoint-cyc} The 1-morphisms $\cal{E}_i \onel$ and $\cal{F}_i \onel$ are biadjoint (up to a specified degree shift). These conditions are expressed diagrammatically as
\begin{equation} \label{eq_biadjoint1-cyc}
    \hackcenter{\begin{tikzpicture}[scale=0.7]
    \draw[thick, <-](0,0) .. controls ++(0,.6) and ++(0,.6) .. (-.75,0) to (-.75,-1);
    \draw[thick, ->](0,0) .. controls ++(0,-.6) and ++(0,-.6) .. (.75,0) to (.75,1);
    \node at (-.5,.9) { $\lambda+\alpha_i$};
    \node at (.8,-.9) { $\lambda$};
    \node at (-.95,-.8) {\tiny $i$};
\end{tikzpicture}}
\;\; = \;\;
    \hackcenter{\begin{tikzpicture}[scale=0.7]
    \draw[thick, ->](0,-1)   to (0,1);
    \node at (-.8,.4) { $\lambda+\alpha_i$};
    \node at (.5,.4) { $\lambda$};
    \node at (-.2,-.8) {\tiny $i$};
\end{tikzpicture}}
\qquad \quad
\hackcenter{\begin{tikzpicture}[scale=0.7]
    \draw[thick, ->](0,0) .. controls ++(0,.6) and ++(0,.6) .. (.75,0) to (.75,-1);
    \draw[thick, <-](0,0) .. controls ++(0,-.6) and ++(0,-.6) .. (-.75,0) to (-.75,1);
    \node at (.5,.9) { $\lambda+\alpha_i$};
    \node at (-.8,-.9) { $\lambda$};
    \node at (-.95,.8) {\tiny $i$};
\end{tikzpicture}}
\;\; = \;\;
    \hackcenter{\begin{tikzpicture}[scale=0.7]
    \draw[thick, <-]  (0,-1) to (0,1);
    \node at (.8,-.4) { $\lambda+\alpha_i$};
    \node at (-.5,-.4) { $\lambda$};
    \node at (-.2,.8) {\tiny $i$};
\end{tikzpicture}}
\end{equation}

\begin{equation} \label{eq_biadjoint2-cyc}
    \hackcenter{\begin{tikzpicture}[scale=0.7]
    \draw[thick, <-](0,0) .. controls ++(0,.6) and ++(0,.6) .. (.75,0) to (.75,-1);
    \draw[thick, ->](0,0) .. controls ++(0,-.6) and ++(0,-.6) .. (-.75,0) to (-.75,1);
    \node at (-.5,-.9) { $\lambda+\alpha_i$};
    \node at (.8,.9) { $\lambda$};
    \node at (.95,-.8) {\tiny $i$};
\end{tikzpicture}}
\;\; = \;\;
    \hackcenter{\begin{tikzpicture}[scale=0.7]
    \draw[thick, ->](0,-1)   to (0,1);
    \node at (-.8,.4) { $\lambda+\alpha_i$};
    \node at (.5,.4) { $\lambda$};
    \node at (-.2,-.8) {\tiny $i$};
\end{tikzpicture}}
\qquad \quad
\hackcenter{\begin{tikzpicture}[scale=0.7]
    \draw[thick, ->](0,0) .. controls ++(0,.6) and ++(0,.6) .. (-.75,0) to (-.75,-1);
    \draw[thick, <-](0,0) .. controls ++(0,-.6) and ++(0,-.6) .. (.75,0) to (.75,1);
    \node at (.5,-.9) { $\lambda+\alpha_i$};
    \node at (-.8,.9) { $\lambda$};
    \node at (.95,.8) {\tiny $i$};
\end{tikzpicture}}
\;\; = \;\;
\hackcenter{\begin{tikzpicture}[scale=0.7]
    \draw[thick, <-]  (0,-1) to (0,1);
    \node at (.8,-.4) { $\lambda+\alpha_i$};
    \node at (-.5,-.4) { $\lambda$};
    \node at (-.2,.8) {\tiny $i$};
\end{tikzpicture}}
\end{equation}

  \item The 2-morphisms are cyclic with respect to this biadjoint structure.
\begin{equation}\label{eq_cyclic_dot-cyc}
\hackcenter{\begin{tikzpicture}[scale=0.7]
    \draw[thick, <-]  (0,-1) to (0,1);
    \node at (.8,-.4) { $\lambda+\alpha_i$};
    \node at (-.5,-.4) { $\lambda$};
    \filldraw  (0,.2) circle (2.5pt);
    \node at (-.2,.8) {\tiny $i$};
\end{tikzpicture}}
\;\; := \;\;\hackcenter{\begin{tikzpicture}[scale=0.7]
    \draw[thick, ->]  (0,.4) .. controls ++(0,.6) and ++(0,.6) .. (-.75,.4) to (-.75,-1);
    \draw[thick, <-](0,.4) to (0,-.4) .. controls ++(0,-.6) and ++(0,-.6) .. (.75,-.4) to (.75,1);
    \filldraw  (0,-.2) circle (2.5pt);
    %\node at (1,-.9) { $\lambda+\alpha_i$};
    \node at (-1,.9) { $\lambda$};
    \node at (.95,.8) {\tiny $i$};
\end{tikzpicture}}
\;\; = \;\;
\hackcenter{\begin{tikzpicture}[scale=0.7]
    \draw[thick, ->]  (0,.4) .. controls ++(0,.6) and ++(0,.6) .. (.75,.4) to (.75,-1);
    \draw[thick, <-](0,.4) to (0,-.4) .. controls ++(0,-.6) and ++(0,-.6) .. (-.75,-.4) to (-.75,1);
    \filldraw  (0,-.2) circle (2.5pt);
    %\node at (1,-.9) { $\lambda+\alpha_i$};
    \node at (1.3,.9) { $\lambda + \alpha_i$};
    \node at (-.95,.8) {\tiny $i$};
\end{tikzpicture}}
\end{equation}

The cyclic relations for crossings are given by
\begin{equation} \label{eq_cyclic}
\hackcenter{
\begin{tikzpicture}[scale=0.7]
    \draw[thick, <-] (0,0) .. controls (0,.5) and (.75,.5) .. (.75,1.0);
    \draw[thick, <-] (.75,0) .. controls (.75,.5) and (0,.5) .. (0,1.0);
    \node at (1.1,.65) { $\lambda$};
    \node at (-.2,.1) {\tiny $i$};
    \node at (.95,.1) {\tiny $j$};
\end{tikzpicture}}
\;\; := \;\;
\hackcenter{\begin{tikzpicture}[scale=0.7]
    \draw[thick, ->] (0,0) .. controls (0,.5) and (.75,.5) .. (.75,1.0);
    \draw[thick, ->] (.75,0) .. controls (.75,.5) and (0,.5) .. (0,1.0);
    \draw[thick] (0,0) .. controls ++(0,-.4) and ++(0,-.4) .. (-.75,0) to (-.75,2);
    \draw[thick] (.75,0) .. controls ++(0,-1.2) and ++(0,-1.2) .. (-1.5,0) to (-1.55,2);
    \draw[thick, ->] (.75,1.0) .. controls ++(0,.4) and ++(0,.4) .. (1.5,1.0) to (1.5,-1);
    \draw[thick, ->] (0,1.0) .. controls ++(0,1.2) and ++(0,1.2) .. (2.25,1.0) to (2.25,-1);
    \node at (-.35,.75) {  $\lambda$};
    \node at (1.3,-.7) {\tiny $i$};
    \node at (2.05,-.7) {\tiny $j$};
    \node at (-.9,1.7) {\tiny $i$};
    \node at (-1.7,1.7) {\tiny $j$};
\end{tikzpicture}}
\quad = \quad
\hackcenter{\begin{tikzpicture}[xscale=-1.0, scale=0.7]
    \draw[thick, ->] (0,0) .. controls (0,.5) and (.75,.5) .. (.75,1.0);
    \draw[thick, ->] (.75,0) .. controls (.75,.5) and (0,.5) .. (0,1.0);
    \draw[thick] (0,0) .. controls ++(0,-.4) and ++(0,-.4) .. (-.75,0) to (-.75,2);
    \draw[thick] (.75,0) .. controls ++(0,-1.2) and ++(0,-1.2) .. (-1.5,0) to (-1.55,2);
    \draw[thick, ->] (.75,1.0) .. controls ++(0,.4) and ++(0,.4) .. (1.5,1.0) to (1.5,-1);
    \draw[thick, ->] (0,1.0) .. controls ++(0,1.2) and ++(0,1.2) .. (2.25,1.0) to (2.25,-1);
    \node at (1.2,.75) {  $\lambda$};
    \node at (1.3,-.7) {\tiny $j$};
    \node at (2.05,-.7) {\tiny $i$};
    \node at (-.9,1.7) {\tiny $j$};
    \node at (-1.7,1.7) {\tiny $i$};
\end{tikzpicture}}
\end{equation}

Sideways crossings are equivalently defined by the following identities:
\begin{equation} \label{eq_crossl-gen-cyc}
\hackcenter{
\begin{tikzpicture}[scale=0.8]
    \draw[thick, ->] (0,0) .. controls (0,.5) and (.75,.5) .. (.75,1.0);
    \draw[thick, <-] (.75,0) .. controls (.75,.5) and (0,.5) .. (0,1.0);
    \node at (1.1,.65) { $\lambda$};
    \node at (-.2,.1) {\tiny $i$};
    \node at (.95,.1) {\tiny $j$};
\end{tikzpicture}}
\;\; := \;\;
\hackcenter{\begin{tikzpicture}[scale=0.7]
    \draw[thick, ->] (0,0) .. controls (0,.5) and (.75,.5) .. (.75,1.0);
    \draw[thick, ->] (.75,-.5) to (.75,0) .. controls (.75,.5) and (0,.5) .. (0,1.0) to (0,1.5);
    \draw[thick] (0,0) .. controls ++(0,-.4) and ++(0,-.4) .. (-.75,0) to (-.75,1.5);
    %\draw[thick] (.75,0) .. controls ++(0,-1.2) and ++(0,-1.2) .. (-1.5,0) to (-1.55,2);
    \draw[thick, ->] (.75,1.0) .. controls ++(0,.4) and ++(0,.4) .. (1.5,1.0) to (1.5,-.5);
    %\draw[thick, ->] (0,1.0) .. controls ++(0,1.2) and ++(0,1.2) .. (2.25,1.0) to (2.25,-1);
    \node at (1.85,.55) {  $\lambda$};
    \node at (1.75,-.2) {\tiny $j$};
    \node at (.55,-.2) {\tiny $i$};
    \node at (-.9,1.2) {\tiny $j$};
    \node at (.25,1.2) {\tiny $i$};
\end{tikzpicture}}
\qquad \qquad
%\quad = \quad
%\hackcenter{\begin{tikzpicture}[xscale=-1.0, scale=0.7]
%    \draw[thick, <-] (0,0) .. controls (0,.5) and (.75,.5) .. (.75,1.0);
%    \draw[thick, <-] (.75,-.5) to (.75,0) .. controls (.75,.5) and (0,.5) .. (0,1.0) to (0,1.5);
%    %
%    \draw[thick,->] (0,0) .. controls ++(0,-.4) and ++(0,-.4) .. (-.75,0) to (-.75,1.5);
%    %\draw[thick] (.75,0) .. controls ++(0,-1.2) and ++(0,-1.2) .. (-1.5,0) to (-1.55,2);
%    \draw[thick] (.75,1.0) .. controls ++(0,.4) and ++(0,.4) .. (1.5,1.0) to (1.5,-.5);
%    %\draw[thick, ->] (0,1.0) .. controls ++(0,1.2) and ++(0,1.2) .. (2.25,1.0) to (2.25,-1);
%    \node at (-1.1,.55) {  $\lambda$};
%    \node at (1.75,-.2) {\tiny $i$};
%    \node at (1,-.2) {\tiny $j$};
%    \node at (-.9,1.2) {\tiny $i$};
%    \node at (.25,1.2) {\tiny $j$};
%\end{tikzpicture}}
%\end{equation}
%\begin{equation}
\hackcenter{
\begin{tikzpicture}[scale=0.8]
    \draw[thick, <-] (0,0) .. controls (0,.5) and (.75,.5) .. (.75,1.0);
    \draw[thick, ->] (.75,0) .. controls (.75,.5) and (0,.5) .. (0,1.0);
    \node at (1.1,.65) { $\lambda$};
    \node at (-.2,.1) {\tiny $i$};
    \node at (.95,.1) {\tiny $j$};
\end{tikzpicture}}
\;\; := \;\;
\hackcenter{\begin{tikzpicture}[xscale=-1.0, scale=0.7]
    \draw[thick, ->] (0,0) .. controls (0,.5) and (.75,.5) .. (.75,1.0);
    \draw[thick, ->] (.75,-.5) to (.75,0) .. controls (.75,.5) and (0,.5) .. (0,1.0) to (0,1.5);
    \draw[thick] (0,0) .. controls ++(0,-.4) and ++(0,-.4) .. (-.75,0) to (-.75,1.5);
    %\draw[thick] (.75,0) .. controls ++(0,-1.2) and ++(0,-1.2) .. (-1.5,0) to (-1.55,2);
    \draw[thick, ->] (.75,1.0) .. controls ++(0,.4) and ++(0,.4) .. (1.5,1.0) to (1.5,-.5);
    %\draw[thick, ->] (0,1.0) .. controls ++(0,1.2) and ++(0,1.2) .. (2.25,1.0) to (2.25,-1);
    \node at (-1.1,.55) {  $\lambda$};
    \node at (1.75,-.2) {\tiny $i$};
    \node at (1,-.2) {\tiny $j$};
    \node at (-.9,1.2) {\tiny $i$};
    \node at (.25,1.2) {\tiny $j$};
\end{tikzpicture}}
%\quad = \quad
%\hackcenter{\begin{tikzpicture}[scale=0.7]
%    \draw[thick, <-] (0,0) .. controls (0,.5) and (.75,.5) .. (.75,1.0);
%    \draw[thick, <-] (.75,-.5) to (.75,0) .. controls (.75,.5) and (0,.5) .. (0,1.0) to (0,1.5);
%    %
%    \draw[thick,->] (0,0) .. controls ++(0,-.4) and ++(0,-.4) .. (-.75,0) to (-.75,1.5);
%    %\draw[thick] (.75,0) .. controls ++(0,-1.2) and ++(0,-1.2) .. (-1.5,0) to (-1.55,2);
%    \draw[thick] (.75,1.0) .. controls ++(0,.4) and ++(0,.4) .. (1.5,1.0) to (1.5,-.5);
%    %\draw[thick, ->] (0,1.0) .. controls ++(0,1.2) and ++(0,1.2) .. (2.25,1.0) to (2.25,-1);
%    \node at (1.85,.55) {  $\lambda$};
%    \node at (1.75,-.2) {\tiny $j$};
%    \node at (.55,-.2) {\tiny $i$};
%    \node at (-.95,1.2) {\tiny $j$};
%    \node at (.25,1.2) {\tiny $i$};
%\end{tikzpicture}}
\end{equation}

\item The $\cal{E}$'s (respectively $\cal{F}$'s) carry an action of the KLR algebra for a fixed choice of parameters $Q$.
The KLR algebra $R_Q$ associated to a fixed set of parameters $Q$ is defined by finite $\Bbbk$-linear combinations of braid--like diagrams in the plane, where each strand is labelled by a vertex $i \in I$.  Strands can intersect and can carry dots, but triple intersections are not allowed.  Diagrams are considered up to planar isotopy that do not change the combinatorial type of the diagram. We recall the local relations:

\begin{enumerate}[i)]

\item The quadratic KLR relations are
\begin{equation}
\hackcenter{
\begin{tikzpicture}[scale=0.8]
    \draw[thick, ->] (0,0) .. controls ++(0,.5) and ++(0,-.4) .. (.75,.8) .. controls ++(0,.4) and ++(0,-.5) .. (0,1.6);
    \draw[thick, ->] (.75,0) .. controls ++(0,.5) and ++(0,-.4) .. (0,.8) .. controls ++(0,.4) and ++(0,-.5) .. (.75,1.6);
    \node at (1.1,1.25) { $\lambda$};
    \node at (-.2,.1) {\tiny $i$};
    \node at (.95,.1) {\tiny $j$};
\end{tikzpicture}}
 \qquad = \qquad
 \left\{
 \begin{array}{ccc}
 0 & & \text{if $i \cdot j=2$, } \\ \\
     t_{ij}\;
     \hackcenter{
\begin{tikzpicture}[scale=0.8]
    \draw[thick, ->] (0,0) to (0,1.6);
    \draw[thick, ->] (.75,0) to (.75,1.6);
    \node at (1.1,1.25) { $\lambda$};
    \node at (-.2,.1) {\tiny $i$};
    \node at (.95,.1) {\tiny $j$};
\end{tikzpicture}}&  &  \text{if $i \cdot j=0$,}\\ \\
% t_{ij} \vcenter{\xy 0;/r.17pc/:
%  (3,9);(3,-9) **\dir{-}?(0)*\dir{<}+(2.3,0)*{};
%  (-3,9);(-3,-9) **\dir{-}?(0)*\dir{<}+(2.3,0)*{};
%  %(8,2)*{\lambda};
%  (-3,4)*{\bullet};(-6.5,5)*{};
%  (-5,-6)*{\scs i};     (5.1,-6)*{\scs j};
% \endxy} \;\; + \;\; t_{ji}
%  \vcenter{\xy 0;/r.17pc/:
%  (3,9);(3,-9) **\dir{-}?(0)*\dir{<}+(2.3,0)*{};
%  (-3,9);(-3,-9) **\dir{-}?(0)*\dir{<}+(2.3,0)*{};
%  %(12,2)*{\lambda};
%  (3,4)*{\bullet};(7,5)*{};
%  (-5,-6)*{\scs i};     (5.1,-6)*{\scs j};
% \endxy}
% %%%%%
  t_{ij}
  \;      \hackcenter{
\begin{tikzpicture}[scale=0.8]
    \draw[thick, ->] (0,0) to (0,1.6);
    \draw[thick, ->] (.75,0) to (.75,1.6);
    \node at (1.1,1.25) { $\lambda$}; \filldraw  (0,.8) circle (2.5pt);
    \node at (-.2,.1) {\tiny $i$};
    \node at (.95,.1) {\tiny $j$};
\end{tikzpicture}}
  \;\; + \;\; t_{ji} \;
 \hackcenter{
\begin{tikzpicture}[scale=0.8]
    \draw[thick, ->] (0,0) to (0,1.6);
    \draw[thick, ->] (.75,0) to (.75,1.6);
    \node at (1.1,1.25) { $\lambda$}; \filldraw  (.75,.8) circle (2.5pt);
    \node at (-.2,.1) {\tiny $i$};
    \node at (.95,.1) {\tiny $j$};
\end{tikzpicture}}
 %%%%
   &  & \text{if $i \cdot j =-1$,}
 \end{array}
 \right. \label{eq_r2_ij-gen-cyc}
\end{equation}

\item The nilHecke dot sliding relations
\begin{align} \label{eq:nil-dot}
\hackcenter{\begin{tikzpicture}[scale=0.8]
    \draw[thick, ->] (0,0) .. controls ++(0,.55) and ++(0,-.5) .. (.75,1)
        node[pos=.25, shape=coordinate](DOT){};
    \draw[thick, ->] (.75,0) .. controls ++(0,.5) and ++(0,-.5) .. (0,1);
    \filldraw  (DOT) circle (2.5pt);
    \node at (-.2,.15) {\tiny $i$};
    \node at (.95,.15) {\tiny $i$};
\end{tikzpicture}}
\quad-\quad
\hackcenter{\begin{tikzpicture}[scale=0.8]
    \draw[thick, ->] (0,0) .. controls ++(0,.55) and ++(0,-.5) .. (.75,1)
        node[pos=.75, shape=coordinate](DOT){};
    \draw[thick, ->] (.75,0) .. controls ++(0,.5) and ++(0,-.5) .. (0,1);
    \filldraw  (DOT) circle (2.5pt);
    \node at (-.2,.15) {\tiny $i$};
    \node at (.95,.15) {\tiny $i$};
\end{tikzpicture}}
\quad=\quad
\hackcenter{\begin{tikzpicture}[scale=0.8]
    \draw[thick, ->] (0,0) .. controls ++(0,.55) and ++(0,-.5) .. (.75,1);
    \draw[thick, ->] (.75,0) .. controls ++(0,.5) and ++(0,-.5) .. (0,1) node[pos=.75, shape=coordinate](DOT){};
    \filldraw  (DOT) circle (2.5pt);
    \node at (-.2,.15) {\tiny $i$};
    \node at (.95,.15) {\tiny $i$};
\end{tikzpicture}}
\quad-\quad
\hackcenter{\begin{tikzpicture}[scale=0.8]
    \draw[thick, ->] (0,0) .. controls ++(0,.55) and ++(0,-.5) .. (.75,1);
    \draw[thick, ->] (.75,0) .. controls ++(0,.5) and ++(0,-.5) .. (0,1) node[pos=.25, shape=coordinate](DOT){};
    \filldraw  (DOT) circle (2.5pt);
    \node at (-.2,.15) {\tiny $i$};
    \node at (.95,.15) {\tiny $i$};
\end{tikzpicture}}
&\quad=\quad
\hackcenter{\begin{tikzpicture}[scale=0.8]
    \draw[thick, ->] (0,0) to  (0,1);
    \draw[thick, ->] (.75,0)to (.75,1) ;
    \node at (-.2,.15) {\tiny $i$};
    \node at (.95,.15) {\tiny $i$};
\end{tikzpicture}}
\end{align}

\item For $i \neq j$ the dot sliding relations
\begin{align} \label{eq:nil-dot}
\hackcenter{\begin{tikzpicture}[scale=0.8]
    \draw[thick, ->] (0,0) .. controls ++(0,.55) and ++(0,-.5) .. (.75,1)
        node[pos=.25, shape=coordinate](DOT){};
    \draw[thick, ->] (.75,0) .. controls ++(0,.5) and ++(0,-.5) .. (0,1);
    \filldraw  (DOT) circle (2.5pt);
    \node at (-.2,.15) {\tiny $i$};
    \node at (.95,.15) {\tiny $j$};
\end{tikzpicture}}
\quad=\quad
\hackcenter{\begin{tikzpicture}[scale=0.8]
    \draw[thick, ->] (0,0) .. controls ++(0,.55) and ++(0,-.5) .. (.75,1)
        node[pos=.75, shape=coordinate](DOT){};
    \draw[thick, ->] (.75,0) .. controls ++(0,.5) and ++(0,-.5) .. (0,1);
    \filldraw  (DOT) circle (2.5pt);
    \node at (-.2,.15) {\tiny $i$};
    \node at (.95,.15) {\tiny $j$};
\end{tikzpicture}}
\quad \qquad
\hackcenter{\begin{tikzpicture}[scale=0.8]
    \draw[thick, ->] (0,0) .. controls ++(0,.55) and ++(0,-.5) .. (.75,1);
    \draw[thick, ->] (.75,0) .. controls ++(0,.5) and ++(0,-.5) .. (0,1) node[pos=.75, shape=coordinate](DOT){};
    \filldraw  (DOT) circle (2.5pt);
    \node at (-.2,.15) {\tiny $i$};
    \node at (.95,.15) {\tiny $j$};
\end{tikzpicture}}
\quad= \quad
\hackcenter{\begin{tikzpicture}[scale=0.8]
    \draw[thick, ->] (0,0) .. controls ++(0,.55) and ++(0,-.5) .. (.75,1);
    \draw[thick, ->] (.75,0) .. controls ++(0,.5) and ++(0,-.5) .. (0,1) node[pos=.25, shape=coordinate](DOT){};
    \filldraw  (DOT) circle (2.5pt);
    \node at (-.2,.15) {\tiny $i$};
    \node at (.95,.15) {\tiny $j$};
\end{tikzpicture}}
\end{align}
hold.

\item Unless $i = k$ and $i \cdot j =-1$ the relation
\begin{align} \label{heis:up-triple}
\hackcenter{\begin{tikzpicture}[scale=0.8]
    \draw[thick, ->] (0,0) .. controls ++(0,1) and ++(0,-1) .. (1.5,2);
    \draw[thick, ] (.75,0) .. controls ++(0,.5) and ++(0,-.5) .. (0,1);
    \draw[thick, ->] (0,1) .. controls ++(0,.5) and ++(0,-.5) .. (0.75,2);
    \draw[thick, ->] (1.5,0) .. controls ++(0,1) and ++(0,-1) .. (0,2);
    \node at (-.2,.15) {\tiny $i$};
    \node at (.95,.15) {\tiny $j$};
    \node at (1.25,.15) {\tiny $k$};
\end{tikzpicture}}
&\;\; = \;\;
\hackcenter{\begin{tikzpicture}[scale=0.8]
    \draw[thick, ->] (0,0) .. controls ++(0,1) and ++(0,-1) .. (1.5,2);
    \draw[thick, ] (.75,0) .. controls ++(0,.5) and ++(0,-.5) .. (1.5,1);
    \draw[thick, ->] (1.5,1) .. controls ++(0,.5) and ++(0,-.5) .. (0.75,2);
    \draw[thick, ->] (1.5,0) .. controls ++(0,1) and ++(0,-1) .. (0,2);
    \node at (-.2,.15) {\tiny $i$};
    \node at (.95,.15) {\tiny $j$};
    \node at (1.75,.15) {\tiny $k$};
\end{tikzpicture}}
\end{align}
holds. Otherwise, $i \cdot j =-1$ and
\begin{equation} \label{eq:KLR-r3}
\hackcenter{\begin{tikzpicture}[scale=0.8]
    \draw[thick, ->] (0,0) .. controls ++(0,1) and ++(0,-1) .. (1.5,2);
    \draw[thick, ] (.75,0) .. controls ++(0,.5) and ++(0,-.5) .. (0,1);
    \draw[thick, ->] (0,1) .. controls ++(0,.5) and ++(0,-.5) .. (0.75,2);
    \draw[thick, ->] (1.5,0) .. controls ++(0,1) and ++(0,-1) .. (0,2);
    \node at (-.2,.15) {\tiny $i$};
    \node at (.95,.15) {\tiny $j$};
    \node at (1.75,.15) {\tiny $i$};
\end{tikzpicture}}
\;\;- \;\;
\hackcenter{\begin{tikzpicture}[scale=0.8]
    \draw[thick, ->] (0,0) .. controls ++(0,1) and ++(0,-1) .. (1.5,2);
    \draw[thick, ] (.75,0) .. controls ++(0,.5) and ++(0,-.5) .. (1.5,1);
    \draw[thick, ->] (1.5,1) .. controls ++(0,.5) and ++(0,-.5) .. (0.75,2);
    \draw[thick, ->] (1.5,0) .. controls ++(0,1) and ++(0,-1) .. (0,2);
    \node at (-.2,.15) {\tiny $i$};
    \node at (.95,.15) {\tiny $j$};
    \node at (1.75,.15) {\tiny $i$};
\end{tikzpicture}}
\;\; = \;\; t_{ij}
\hackcenter{\begin{tikzpicture}[scale=0.8]
    \draw[thick, ->] (0,0) to (0,2);
    \draw[thick, -> ] (.75,0) to (0.75,2);
    \draw[thick, ->] (1.5,0) to (1.5,2);
    \node at (-.2,.15) {\tiny $i$};
    \node at (.95,.15) {\tiny $j$};
    \node at (1.75,.15) {\tiny $i$};
\end{tikzpicture}}
\end{equation}
\end{enumerate}

\item When $i \ne j$ one has the mixed relations  relating $\cal{E}_i \cal{F}_j$ and $\cal{F}_j \cal{E}_i$:
\begin{equation}  \label{mixed_rel-cyc}
 \hackcenter{\begin{tikzpicture}[scale=0.8]
    \draw[thick,<-] (0,0) .. controls ++(0,.5) and ++(0,-.5) .. (.75,1);
    \draw[thick] (.75,0) .. controls ++(0,.5) and ++(0,-.5) .. (0,1);
    \draw[thick, ->] (0,1 ) .. controls ++(0,.5) and ++(0,-.5) .. (.75,2);
    \draw[thick] (.75,1) .. controls ++(0,.5) and ++(0,-.5) .. (0,2);
        \node at (-.2,.15) {\tiny $j$};
    \node at (.95,.15) {\tiny $i$};
\end{tikzpicture}}
\;\; = \;\;
\hackcenter{\begin{tikzpicture}[scale=0.8]
    \draw[thick, <-] (0,0) -- (0,2);
    \draw[thick, ->] (.75,0) -- (.75,2);
     \node at (-.2,.2) {\tiny $j$};
    \node at (.95,.2) {\tiny $i$};
\end{tikzpicture}}
\qquad \qquad
 \hackcenter{\begin{tikzpicture}[scale=0.8]
    \draw[thick] (0,0) .. controls ++(0,.5) and ++(0,-.5) .. (.75,1);
    \draw[thick, <-] (.75,0) .. controls ++(0,.5) and ++(0,-.5) .. (0,1);
    \draw[thick] (0,1 ) .. controls ++(0,.5) and ++(0,-.5) .. (.75,2);
    \draw[thick, ->] (.75,1) .. controls ++(0,.5) and ++(0,-.5) .. (0,2);
        \node at (-.2,.15) {\tiny $i$};
    \node at (.95,.15) {\tiny $j$};
\end{tikzpicture}}
\;\; = \;\;
\hackcenter{\begin{tikzpicture}[scale=0.8]
    \draw[thick, ->] (0,0) -- (0,2);
    \draw[thick, <-] (.75,0) -- (.75,2);
     \node at (-.2,.2) {\tiny $i$};
    \node at (.95,.2) {\tiny $j$};
\end{tikzpicture}}
\end{equation}

\item Negative degree bubbles are zero.  That is for all $m \in \Z_{>0}$ one has
\begin{equation}
 \hackcenter{ \begin{tikzpicture} [scale=.8]
 \draw (-.15,.35) node { $\scs i$};
 \draw[ ]  (0,0) arc (180:360:0.5cm) [thick];
 \draw[<- ](1,0) arc (0:180:0.5cm) [thick];
\filldraw  [black] (.1,-.25) circle (2.5pt);
 \node at (-.2,-.5) {\tiny $m$};
 \node at (1.15,.8) { $\lambda  $};
\end{tikzpicture} } \;  = 0\quad \text{if $m < \l_i -1$}, \qquad \quad
\;
\hackcenter{ \begin{tikzpicture} [scale=.8]
 \draw (-.15,.35) node { $\scs i$};
 \draw  (0,0) arc (180:360:0.5cm) [thick];
 \draw[->](1,0) arc (0:180:0.5cm) [thick];
\filldraw  [black] (.9,-.25) circle (2.5pt);
 \node at (1,-.5) {\tiny $m$};
 \node at (1.15,.8) { $\lambda $};
\end{tikzpicture} } \;  = 0 \quad  \text{if $m < -\l_i -1$}
\end{equation}
Furthermore, dotted bubbles of degree zero are scalar multiples of the identity 2-morphism determined by the bubble parameters
\begin{equation} \label{eq:degreezero}
 \hackcenter{ \begin{tikzpicture} [scale=.8]
 \draw (-.15,.35) node { $\scs i$};
 \draw  (0,0) arc (180:360:0.5cm) [thick];
 \draw[,<-](1,0) arc (0:180:0.5cm) [thick];
\filldraw  [black] (.1,-.25) circle (2.5pt);
 \node at (-.5,-.5) {\tiny $\l_i -1$};
 \node at (1.15,1) { $\lambda  $};
\end{tikzpicture} }
\;\; =
\sfc_{i,\l}^+ \cdot  \Id_{\1_{\l}}
\quad \text{for $  \l_i \geq 1$}, \qquad \quad
\;
\hackcenter{ \begin{tikzpicture} [scale=.8]
 \draw (-.15,.35) node { $\scs i$};
 \draw  (0,0) arc (180:360:0.5cm) [thick];
 \draw[->](1,0) arc (0:180:0.5cm) [thick];
\filldraw  [black] (.9,-.25) circle (2.5pt);
 \node at (1.35,-.5) {\tiny $-\l_i -1$};
 \node at (1.15,1) { $\lambda $};
\end{tikzpicture} }
\;\; =
\sfc_{i,\l}^{-}  \cdot \Id_{\1_{\l}} \quad \text{if $   \l_i \leq -1$.}
\end{equation}

 We introduce formal symbols called \emph{fake bubbles}.  These are positive degree endomorphisms of $\onel$ that carry a formal label by a negative number of dots.

\begin{itemize}
  \item Degree zero fake bubbles are equal to
  \begin{equation}
 \hackcenter{ \begin{tikzpicture} [scale=.8]
 \draw (-.15,.35) node { $\scs i$};
 \draw  (0,0) arc (180:360:0.5cm) [thick];
 \draw[,<-](1,0) arc (0:180:0.5cm) [thick];
\filldraw  [black] (.1,-.25) circle (2.5pt);
 \node at (-.5,-.55) {\tiny $\l_i -1$};
 \node at (1.15,1) { $\lambda  $};
\end{tikzpicture} } \;\; =  \sfc_{i,\l}^+\cdot \Id_{\1_{\l}} \quad \text{for $  \l_i \leq 0$}, \qquad \quad
\;
\hackcenter{ \begin{tikzpicture} [scale=.8]
 \draw (-.15,.35) node { $\scs i$};
 \draw  (0,0) arc (180:360:0.5cm) [thick];
 \draw[->](1,0) arc (0:180:0.5cm) [thick];
\filldraw  [black] (.9,-.25) circle (2.5pt);
 \node at (1.35,-.5) {\tiny $-\l_i -1$};
 \node at (1.15,1) { $\lambda $};
\end{tikzpicture} } \;\; = \sfc_{i,\l}^{-}\cdot \Id_{\1_{\l}} \quad \text{if $   \l_i \geq 0$}
\end{equation}
(compare with \eqref{eq:degreezero}).
\item Higher degree fake bubbles for $\l_i <0$ are defined inductively as
\begin{equation} \label{eq:fake1}
  \hackcenter{ \begin{tikzpicture} [scale=.8]
 \draw (-.15,.35) node { $\scs i$};
 \draw  (0,0) arc (180:360:0.5cm) [thick];
 \draw[,<-](1,0) arc (0:180:0.5cm) [thick];
\filldraw  [black] (.1,-.25) circle (2.5pt);
 \node at (-.65,-.55) {\tiny $\l_i -1+j$};
 \node at (1.15,1) { $\lambda  $};
\end{tikzpicture} } \;\; = \;\;
\left\{
  \begin{array}{ll}
    \textcolor[rgb]{1.00,0.00,0.00}{- \frac{1}{ \sfc_{i,\l}^{-}} }
\displaystyle \sum_{\stackrel{\scs x+y=j}{\scs y\geq 1}} \hackcenter{ \begin{tikzpicture}[scale=.8]
 \draw (-.15,.35) node { $\scs i$};
 \draw  (0,0) arc (180:360:0.5cm) [thick];
 \draw[,<-](1,0) arc (0:180:0.5cm) [thick];
\filldraw  [black] (.1,-.25) circle (2.5pt);
 \node at (-.35,-.45) {\tiny $\overset{\l_i-1}{+x}$};
 \node at (.85,1) { $\lambda$};
\end{tikzpicture}  \;\;
\begin{tikzpicture}[scale=.8]
 \draw (-.15,.35) node { $\scs i$};
 \draw  (0,0) arc (180:360:0.5cm) [thick];
 \draw[->](1,0) arc (0:180:0.5cm) [thick];
\filldraw  [black] (.9,-.25) circle (2.5pt);
 \node at (1.45,-.5) {\tiny $\overset{-\l_i-1}{+y}$};
 \node at (1.15,1.1) { $\;$};
\end{tikzpicture}  }
  & \hbox{if $0 < j < -\l_i+1$;} \\
    0, & \hbox{if $j<0$.}
  \end{array}
\right.
\end{equation}

\item Higher degree fake bubbles for $\l_i >0$ are defined inductively as
\begin{equation}\label{eq:fake2}
\hackcenter{ \begin{tikzpicture} [scale=.8]
 \draw (-.15,.35) node { $\scs i$};
 \draw  (0,0) arc (180:360:0.5cm) [thick];
 \draw[->](1,0) arc (0:180:0.5cm) [thick];
\filldraw  [black] (.9,-.25) circle (2.5pt);
 \node at (1.5,-.5) {\tiny $-\l_i -1+j$};
 \node at (1.15,1) { $\lambda $};
\end{tikzpicture} } \;\; = \;\;
\left\{
  \begin{array}{ll}
    \textcolor[rgb]{1.00,0.00,0.00}{- \frac{1}{  \sfc_{i,\l}^{+} } }
\displaystyle\sum_{\stackrel{\scs x+y=j}{\scs x\geq 1}} \hackcenter{ \begin{tikzpicture}[scale=.8]
 \draw (-.15,.35) node { $\scs i$};
 \draw  (0,0) arc (180:360:0.5cm) [thick];
 \draw[,<-](1,0) arc (0:180:0.5cm) [thick];
\filldraw  [black] (.1,-.25) circle (2.5pt);
 \node at (-.35,-.45) {\tiny $\overset{\l_i-1}{+x}$};
 \node at (.85,1) { $\lambda$};
\end{tikzpicture}  \;\;
\begin{tikzpicture}[scale=.8]
 \draw (.3,.125) node {};
 \draw  (0,0) arc (180:360:0.5cm) [thick];
 \draw[->](1,0) arc (0:180:0.5cm) [thick];
\filldraw  [black] (.9,-.25) circle (2.5pt);
 \node at (1.45,-.5) {\tiny $\overset{-\l_i-1}{+y}$};
 \node at (1.15,1.1) { $\;$};
\end{tikzpicture}  }
  & \hbox{if $0 < j < \l_i+1$;} \\
    0, & \hbox{if $j<0$.}
  \end{array}
\right.
\end{equation}
\end{itemize}
The above relations are sometimes referred to as the \textit{infinite Grassmannian relations} which are given by comparing coefficients of $t$ in the expression\footnote{Here the product formula for the infinite Grassmannian relation is equal to $-\frac{1}{\beta_{i,\l}}$ rather than $1$ as is usually the case.    This is one motivation for considering all $\beta_{i,\l}=-1$ as a natural choice.}:
\begin{align}
  \left(
  \hackcenter{ \begin{tikzpicture} [scale=.6]
    % Bubble
 \draw  (-.75,1) arc (360:180:.45cm) [thick];
 \draw[<-](-.75,1) arc (0:180:.45cm) [thick];
     \filldraw  [black] (-1.55,.75) circle (2.5pt);
        \node at (-1.3,.3) { $\scriptstyle \l_i -1$};
        \node at (-1.4,1.7) { $i $};
 %% L
 \node at (-.2,1.5) { $\lambda $};
\end{tikzpicture}} \; t^0 \; + \;
  \hackcenter{ \begin{tikzpicture} [scale=.6]
    % Bubble
 \draw  (-.75,1) arc (360:180:.45cm) [thick];
 \draw[<-](-.75,1) arc (0:180:.45cm) [thick];
     \filldraw  [black] (-1.55,.75) circle (2.5pt);
        \node at (-1.3,.3) { $\scriptstyle \l_i -1+1$};
        \node at (-1.4,1.7) { $i $};
 %% L
 \node at (-.2,1.5) { $\lambda $};
\end{tikzpicture}} \; t^1
\;\; + \;\; \dots \;\; + \;\;
  \hackcenter{ \begin{tikzpicture} [scale=.6]
    % Bubble
 \draw  (-.75,1) arc (360:180:.45cm) [thick];
 \draw[<-](-.75,1) arc (0:180:.45cm) [thick];
     \filldraw  [black] (-1.55,.75) circle (2.5pt);
        \node at (-1.3,.3) { $\scriptstyle \l_i -1 + r$};
        \node at (-1.4,1.7) { $i $};
 %% L
 \node at (-.2,1.5) { $\lambda $};
\end{tikzpicture}} \; t^r
\;\; + \;\; \dots
  \right) \times  \hspace{1in} \nn \\
  \left(
  \hackcenter{ \begin{tikzpicture} [scale=.6]
    % Bubble
 \draw  (-.75,1) arc (360:180:.45cm) [thick];
 \draw[->](-.75,1) arc (0:180:.45cm) [thick];
     \filldraw  [black] (-1.55,.75) circle (2.5pt);
        \node at (-1.3,.3) { $\scriptstyle -\l_i -1$};
        \node at (-1.4,1.7) { $i $};
 %% L
 \node at (-.2,1.5) { $\lambda $};
\end{tikzpicture}} \; t^0 \; + \;
\hackcenter{ \begin{tikzpicture} [scale=.6]
    % Bubble
 \draw  (-.75,1) arc (360:180:.45cm) [thick];
 \draw[->](-.75,1) arc (0:180:.45cm) [thick];
     \filldraw  [black] (-1.55,.75) circle (2.5pt);
        \node at (-1.3,.3) { $\scriptstyle -\l_i -1 + 1$};
        \node at (-1.4,1.7) { $i $};
 %% L
 \node at (-.2,1.5) { $\lambda $};
\end{tikzpicture}}\; t
\; + \; \dots \; + \;
\hackcenter{ \begin{tikzpicture} [scale=.6]
    % Bubble
 \draw  (-.75,1) arc (360:180:.45cm) [thick];
 \draw[->](-.75,1) arc (0:180:.45cm) [thick];
     \filldraw  [black] (-1.55,.75) circle (2.5pt);
        \node at (-1.3,.3) { $\scriptstyle -\l_i -1 + s$};
        \node at (-1.4,1.7) { $i $};
 %% L
 \node at (-.2,1.5) { $\lambda $};
\end{tikzpicture}}\; t^s
\; + \; \dots
  \right)
  = \textcolor[rgb]{1.00,0.00,0.00}{-\frac{1}{\beta_{i,\l}}}.
\end{align}

\item The $\mf{sl}_2$ relations (which we also refer to as the $\sE \sF$ and $\sF \sE$ decompositions) are:
\begin{equation}
\begin{split}
\hackcenter{\begin{tikzpicture}[scale=0.8]
    \draw[thick, ->] (0,0) -- (0,2);
    \draw[thick, <-] (.75,0) -- (.75,2);
     \node at (-.2,.2) {\tiny $i$};
    \node at (.95,.2) {\tiny $i$};
 \node at (1.1,1.44) { $\lambda $};
\end{tikzpicture}}
\;\; = \;\; \textcolor[rgb]{1.00,0.00,0.00}{\beta_{i,\l}} \;
 \hackcenter{\begin{tikzpicture}[scale=0.8]
    \draw[thick] (0,0) .. controls ++(0,.5) and ++(0,-.5) .. (.75,1);
    \draw[thick,<-] (.75,0) .. controls ++(0,.5) and ++(0,-.5) .. (0,1);
    \draw[thick] (0,1 ) .. controls ++(0,.5) and ++(0,-.5) .. (.75,2);
    \draw[thick, ->] (.75,1) .. controls ++(0,.5) and ++(0,-.5) .. (0,2);
        \node at (-.2,.15) {\tiny $i$};
    \node at (.95,.15) {\tiny $i$};
\node at (1.1,1.44) { $\lambda $};
\end{tikzpicture}}
\;\; - \;\; \textcolor[rgb]{1.00,0.00,0.00}{\beta_{i,\l}} \;
\sum_{\overset{f_1+f_2+f_3}{=\l_i-1}}\hackcenter{
 \begin{tikzpicture}[scale=0.8]
 \draw[thick,->] (0,-1.0) .. controls ++(0,.5) and ++ (0,.5) .. (.8,-1.0) node[pos=.75, shape=coordinate](DOT1){};
  \draw[thick,<-] (0,1.0) .. controls ++(0,-.5) and ++ (0,-.5) .. (.8,1.0) node[pos=.75, shape=coordinate](DOT3){};
 \draw[thick,->] (0,0) .. controls ++(0,-.45) and ++ (0,-.45) .. (.8,0)node[pos=.25, shape=coordinate](DOT2){};
 \draw[thick] (0,0) .. controls ++(0,.45) and ++ (0,.45) .. (.8,0);
 \draw (-.15,.7) node { $\scs i$};
\draw (1.05,0) node { $\scs i$};
\draw (-.15,-.7) node { $\scs i$};
% \draw (.3,.125) node {};
% \draw  (0,0) arc (180:360:0.5cm) [thick];
% \draw[,<-](1,0) arc (0:180:0.5cm) [thick];
%\filldraw  [black] (.1,-.25) circle (2.5pt);
 \node at (.95,.65) {\tiny $f_3$};
 \node at (-.55,-.05) {\tiny $\overset{-\l_i-1}{+f_2}$};
  \node at (.95,-.65) {\tiny $f_1$};
 \node at (1.55,.3) { $\lambda $};
 \filldraw[thick]  (DOT3) circle (2.5pt);
  \filldraw[thick]  (DOT2) circle (2.5pt);
  \filldraw[thick]  (DOT1) circle (2.5pt);
\end{tikzpicture} }
\\
\hackcenter{\begin{tikzpicture}[scale=0.8]
    \draw[thick, <-] (0,0) -- (0,2);
    \draw[thick, ->] (.75,0) -- (.75,2);
     \node at (-.2,.2) {\tiny $i$};
    \node at (.95,.2) {\tiny $i$};
 \node at (1.1,1.44) { $\lambda $};
\end{tikzpicture}}
\;\; = \;\; \textcolor[rgb]{1.00,0.00,0.00}{\beta_{i,\l}} \;
 \hackcenter{\begin{tikzpicture}[scale=0.8]
    \draw[thick,<-] (0,0) .. controls ++(0,.5) and ++(0,-.5) .. (.75,1);
    \draw[thick] (.75,0) .. controls ++(0,.5) and ++(0,-.5) .. (0,1);
    \draw[thick, ->] (0,1 ) .. controls ++(0,.5) and ++(0,-.5) .. (.75,2);
    \draw[thick] (.75,1) .. controls ++(0,.5) and ++(0,-.5) .. (0,2);
        \node at (-.2,.15) {\tiny $i$};
    \node at (.95,.15) {\tiny $i$};
 \node at (1.1,1.44) { $\lambda $};
\end{tikzpicture}}
\;\; - \;\; \textcolor[rgb]{1.00,0.00,0.00}{\beta_{i,\l}} \;
\sum_{\overset{f_1+f_2+f_3}{=-\l_i-1}}\hackcenter{
 \begin{tikzpicture}[scale=0.8]
 \draw[thick,<-] (0,-1.0) .. controls ++(0,.5) and ++ (0,.5) .. (.8,-1.0) node[pos=.75, shape=coordinate](DOT1){};
  \draw[thick,->] (0,1.0) .. controls ++(0,-.5) and ++ (0,-.5) .. (.8,1.0) node[pos=.75, shape=coordinate](DOT3){};
 \draw[thick ] (0,0) .. controls ++(0,-.45) and ++ (0,-.45) .. (.8,0)node[pos=.25, shape=coordinate](DOT2){};
 \draw[thick, ->] (0,0) .. controls ++(0,.45) and ++ (0,.45) .. (.8,0);
 \draw (-.15,.7) node { $\scs i$};
\draw (1.05,0) node { $\scs i$};
\draw (-.15,-.7) node { $\scs i$};
% \draw (.3,.125) node {};
% \draw  (0,0) arc (180:360:0.5cm) [thick];
% \draw[,<-](1,0) arc (0:180:0.5cm) [thick];
%\filldraw  [black] (.1,-.25) circle (2.5pt);
 \node at (.95,.65) {\tiny $f_3$};
 \node at (-.55,-.05) {\tiny $\overset{\l_i-1}{+f_2}$};
  \node at (.95,-.65) {\tiny $f_1$};
 \node at (1.55,.3) { $\lambda $};
 \filldraw[thick]  (DOT3) circle (2.5pt);
  \filldraw[thick]  (DOT2) circle (2.5pt);
  \filldraw[thick]  (DOT1) circle (2.5pt);
\end{tikzpicture} }  \label{eq:modEF}
\end{split}
\end{equation}
\end{enumerate}
\end{definition}

It is sometimes convenient to use a shorthand notation for the bubbles that emphasizes their degrees.
\begin{equation}
\label{starnotation}
\hackcenter{ \begin{tikzpicture} [scale=.8]
    % Bubble
 \draw  (-.75,1) arc (360:180:.45cm) [thick];
 \draw[<-](-.75,1) arc (0:180:.45cm) [thick];
     \filldraw  [black] (-1.55,.75) circle (2.5pt);
        \node at (-1.3,.3) { $\scriptstyle \ast + r$};
        \node at (-1.4,1.7) { $i $};
 %% L
 \node at (-.2,1.5) { $\lambda $};
\end{tikzpicture}}
\;\; := \;\;
\hackcenter{ \begin{tikzpicture} [scale=.8]
    % Bubble
 \draw  (-.75,1) arc (360:180:.45cm) [thick];
 \draw[<-](-.75,1) arc (0:180:.45cm) [thick];
     \filldraw  [black] (-1.55,.75) circle (2.5pt);
        \node at (-1.3,.3) { $\scriptstyle \l_i -1 + r$};
        \node at (-1.4,1.7) { $i $};
 %% L
 \node at (-.2,1.5) { $\lambda $};
\end{tikzpicture}}
\qquad
\quad
\hackcenter{ \begin{tikzpicture} [scale=.8]
    % Bubble
 \draw  (-.75,1) arc (360:180:.45cm) [thick];
 \draw[->](-.75,1) arc (0:180:.45cm) [thick];
     \filldraw  [black] (-1.55,.75) circle (2.5pt);
        \node at (-1.3,.3) { $\scriptstyle \ast + r$};
        \node at (-1.4,1.7) { $i $};
 %% L
 \node at (-.2,1.5) { $\lambda $};
\end{tikzpicture}}
\;\; := \;\;
\hackcenter{ \begin{tikzpicture} [scale=.8]
    % Bubble
 \draw  (-.75,1) arc (360:180:.45cm) [thick];
 \draw[->](-.75,1) arc (0:180:.45cm) [thick];
     \filldraw  [black] (-1.55,.75) circle (2.5pt);
        \node at (-1.3,.3) { $\scriptstyle -\l_i -1 + r$};
        \node at (-1.4,1.7) { $i $};
 %% L
 \node at (-.2,1.5) { $\lambda $};
\end{tikzpicture}}
\end{equation}

% - - - - - - - - - - - - -  - - - - - - - - - - - - -
\subsection{The usual cyclic form of the 2-category }
% - - - - - - - - - - - - -  - - - - - - - - - - - - -

\begin{definition}
The  cyclic 2-category $\cal{U}_{Q}^{cyc}(\mf{g})$ defined in \cite{BHLW2} is a specialization of $\cal{U}_{Q,\beta}(\mf{g})$ corresponding to taking $\beta_{i,\l}=-1$ for all $i \in I$ and $\l \in X$.  In that case, $\sfc_{i,\l}^{-}=(\sfc_{i,\l}^{+})^{-1}$ and we will use the notation $c_{i,\l}:= \sfc_{i,\l}^+$, so that $c_{i,\l}^{-1} = \sfc_{i,\l}^-$.   Since $\beta_{i,\l}=-1$, it follows that the coefficients $c_{i,\l}$ are constant along $\mf{sl}_2$ strings $c_{i,\l} = c_{i,\l \pm k \alpha_i}$ for all $k$.
\end{definition}

\begin{remark}
  The original version of the 2-category $\cal{U}(\mf{g})$ from \cite{KL3,CLau} is not a specialization of $\cal{U}_{Q,\beta}(\mf{g})$ (outside of $\mf{sl}_2$) since these 2-categories are all cyclic.  However, \cite[Theorem 2.1]{BHLW2} gives an explicit isomorphism of categories $\cal{M}\maps \cal{U}_Q^{cyc}(\mf{g}) \to \cal{U}(\mf{g})$ given by
\begin{align}
   \cal{M}\left(\;\;
\hackcenter{\begin{tikzpicture}[scale=0.8]
    \draw[thick, <-] (.75,-2) .. controls ++(0,.75) and ++(0,.75) .. (0,-2);
    \node at (.4,-1.2) {\tiny $\lambda$};
    \node at (-.2,-1.9) {\tiny $i$};
\end{tikzpicture}}
\;\;\right)
\;\;  = \;\; c_{i,\l}
    \hackcenter{\begin{tikzpicture}[scale=0.8, blue]
    \draw[thick, <-] (.75,-2) .. controls ++(0,.75) and ++(0,.75) .. (0,-2);
    \node at (.4,-1.2) {\tiny $\lambda$};
    \node at (-.2,-1.9) {\tiny $i$};
\end{tikzpicture}}
\qquad \qquad
\cal{M} \left(
\;\; \hackcenter{\begin{tikzpicture}[scale=0.8]
    \draw[thick, <-] (.75,2) .. controls ++(0,-.75) and ++(0,-.75) .. (0,2);
    \node at (.4,1.2) {\tiny $\lambda$};
    \node at (-.2,1.9) {\tiny $i$};
\end{tikzpicture}}
\;\; \right)
\;\;  = \;\; c_{i,\l}^{-1}
    \hackcenter{\begin{tikzpicture}[scale=0.8, blue]
    \draw[thick, <-] (.75,2) .. controls ++(0,-.75) and ++(0,-.75) .. (0,2);
    \node at (.4,1.2) {\tiny $\lambda$};
    \node at (-.2,1.9) {\tiny $i$};
\end{tikzpicture}}
\end{align}
so that a sideways crossing is rescaled by
\begin{equation}
 \cal{M} \left(  \hackcenter{\begin{tikzpicture}[scale=0.8]
    \draw[thick, ->] (0,0) .. controls (0,.5) and (.75,.5) .. (.75,1.0);
    \draw[thick, <-] (.75,0) .. controls (.75,.5) and (0,.5) .. (0,1.0);
    \node at (1.1,.55) {\tiny $\lambda$};
    \node at (-.2,.1) {\tiny $j$};
    \node at (.95,.1) {\tiny $i$};
\end{tikzpicture}} \right)
\;\;   = \;\; t_{ij}^{-1}
   \hackcenter{\begin{tikzpicture}[scale=0.8, blue]
    \draw[thick, ->] (0,0) .. controls (0,.5) and (.75,.5) .. (.75,1.0);
    \draw[thick, <-] (.75,0) .. controls (.75,.5) and (0,.5) .. (0,1.0);
    \node at (1.1,.55) {\tiny $\lambda$};
    \node at (-.2,.1) {\tiny $j$};
    \node at (.95,.1) {\tiny $i$};
\end{tikzpicture}}
\end{equation}
\end{remark}

% - - - - - - - - - - - - -  - - - - - - - - - - - - -
\subsection{The $\mf{gl}(n)$-version of the 2-category }
% - - - - - - - - - - - - -  - - - - - - - - - - - - -

The weight lattice of $U_q(\mf{gl}_n)$ is given by tuples $\l = (\l_1, \dots, \l_n)\in \Z^n$.  Such a weight determines an $\mf{sl}_n$ weight $\bar{\l}=(\bar{\l}_1, \dots, \bar{\l}_{n-1})\in \Z^{n-1}$ such that $\bar{\l}_k := \l_k -\l_{k+1}$. Let $\epsilon_i=(0,\dots,0,1,0,\dots,0) \in \Z^n$ with a single 1 in the $i$th position. Let $\alpha_i = \epsilon_i - \epsilon_{i+1}\in \Z^n$ for $1 \leq i \leq n-1$.

Here we will be interested in starting from a $\mf{sl}_n$ weight $\mu = (\mu_1, \dots, \mu_{n-1})$ and producing a $\mf{gl}_n$ weight, though such an assignment is not unique.  For $d \in \Z$ define a map
\begin{equation}
  \phi_{n,d} \maps \Z^{n-1} \longrightarrow \Z^{n} \cup \{ \ast \}
\end{equation}
by sending $\phi_{n,d}(\mu) = \lambda$, where $\l$ uniquely satisfies
\begin{align}
  \l_i - \l_{i+1} = \mu_i
\\
\sum_{i=1}^n \l_i = d
\end{align}
if such a solution exists and $\phi_{n,d}(\mu) = \ast$ otherwise.
Write $(\Z^n)_d$ for the set of weights $\l \in \Z^n$ such that $\sum_{i} \lambda_i = d$.  Let $\dot{U}_q(\mf{gl}_n)$ denote Lusztig's idempotent form of the quantum group $U_q(\mf{gl}_n)$ and set
$\dot{U}_q(\mf{gl}_n)^d$ to be the subalgebra of $\dot{U}_q(\mf{gl}_n)$ spanned weights $\l \in (\Z^n)_d$.

For a given $\mu$, its image $\phi_{n,d}(\mu)$ has a solution if and only if
\[
d \equiv \sum_{i=1}^{n-1} i \mu_{i} \quad \mod n.
\]
In particular, for a fixed $d\in \Z$, then exactly one  of the maps $\phi_{n,d+k}$ will have a solution for $k=0,1,\dots, n-1$.
Hence, for all $d\in \Z$ there is an isomorphism\footnote{Thanks to Hoel Queffelec for pointing out this fact. }
%\begin{align}
%   \Z^{n-1} & \longrightarrow \;\; \Z^{n}_{d} \otimes \Z^{n}_{d+1} \otimes \dots \otimes \Z^{n}_{d+n-1} \\
%    \mu &\mapsto \;\;  \phi_{n,d+k} \quad \text{if $$}
%\left(\phi_{n,d}(\mu), \phi_{n,d+1}(\mu), \dots, \phi_{n,d+n-1}(\mu)\right).
%\end{align}
\begin{equation} \label{eq:sln2gln}
  \dot{U}_q(\mf{sl}_n) \cong  \bigoplus_{k=d}^{d+n-1}\dot{U}_q(\mf{gl}_n)^k
\end{equation}
determined on weights by sending $\mu$ to $\phi_{n,d+k}(\mu)$ where $d+k \equiv \sum_{i=1}^{n-1} i \mu_{i} \mod n$, see for example \cite{MSV} or \cite[Remark 2.2]{QRS-annular} and the references therein.

\begin{definition}
Fix a compatible choice of scalars and bubble coefficients $Q$ and $\beta$.   The 2-category $\cal{U}_{Q,\beta}(\mf{gl}_n)$ is the graded additive 2-category with
\begin{itemize}
\item \textbf{Objects} objects $\lambda  \in \Z^n$,
\item \textbf{1-morphisms} are formal direct sums of (shifts of) compositions of
$$\onel, \quad \1_{\lambda+\alpha_i} \sE_i= \1_{\lambda+\alpha_i} \sE_i\onel, \quad \text{ and }\quad
\1_{\lambda-\alpha_i} \sF_i= \1_{\lambda-\alpha_i} \sF_i\onel$$
for $1 \leq i \leq n-1$ and $\lambda \in \Z^{n}$.  We denote the grading shift by $\la 1 \ra$, so that for each 1-morphism $x$ in $\cal{U}$ and $t\in \Z$ we a 1-morphism $x\la t\ra$.

\item \textbf{2-morphisms} the same as those in the 2-category $\cal{U}_{Q,\beta}(\mf{sl}_n)$
\end{itemize}
and relations identical to those in $\cal{U}_{Q,\beta}(\mf{sl}_n)$ with each $\l_i := \la i,\l\ra$ for $\l \in \Z^{n-1}$ replaced by $\bar{\l}_i$.
\end{definition}

Then the isomorphism \eqref{eq:sln2gln} motivates the following 2-functors defined for each $d\in \Z$
\begin{align}
 \phi_{n,d} \maps \cal{U}_{Q,\beta}(\mf{sl}_n) & \longrightarrow \cal{U}_{Q,\beta}(\mf{gl}_n)
 \\
 \mu & \mapsto \phi_{n,d}(\mu)
 \nn \\
 \cal{E}_i \1_{\mu} & \mapsto \cal{E}_i \1_{\phi_{n,d}(\mu)} \nn
\end{align}
which send the generating 1-morphisms and 2-morphisms of $\cal{U}_{Q,\beta}(\mf{sl}_n)$ to the corresponding ones in $\cal{U}_{Q,\beta}(\mf{gl}_n)$.
Using $\mf{gl}_n$ weights it is possible to specify an explicit choice of compatible scalars and bubble coefficients.

\begin{definition}
The 2-category $\cal{U}(\mf{gl}_n)$ defined in \cite{MSV} corresponds to the  choice of scalars
\[
t_{ij} =
\left\{
  \begin{array}{ll}
      -1 & \hbox{if $i=j$} \\
     -1 & \hbox{ if $i \cdot j =-1$ and $i \longrightarrow j$} \\
      1 & \hbox{ otherwise, }
  \end{array}
\right.
\]
(so that $\beta_{i,\l} = -t_{ii}=1$) and bubble parameters
\[
\sfc_{i,\l}^+ = (-1)^{\l_{i+1}} \quad \text{for $\bar{\l}_i \geq 0$},
\qquad
\sfc_{i,\l}^- = (-1)^{\l_{i+1}-1} \quad \text{for $\bar{\l}_i \leq 0$}.
\]
\end{definition}

Hence, the \cite{MSV} 2-category $\cal{U}(\mf{gl}_n)$  is a specialization of $\cal{U}_{Q,\beta}(\mf{gl}_n)$.

% ====================================================
%
\section{Isomorphism of 2-categories}
%
% ====================================================

To aid the reader, our convention throughout this section is to indicate diagrams in the image of a 2-isomorphism in \textcolor[rgb]{0.00,0.07,1.00}{blue}.

% - - - - - - - - - - - - -  - - - - - - - - - - - - -
\subsection{Cyclic to general 2-isomorphisms}
% - - - - - - - - - - - - -  - - - - - - - - - - - - -

Recall that the cyclic 2-category $\cal{U}_{Q}^{cyc}(\mf{g})$ is a specialization of $\cal{U}_{Q,\beta}(\mf{g})$ corresponding to taking $\beta_{i,\l}=-1$ for all $i \in I$ and $\l \in X$.   We denote the bubble parameters in $\cal{U}_{Q}^{cyc}(\mf{g})$ using the notation $c_{i,\l}:= \sfc_{i,\l}^+$ and $c_{i,\l}^{-1} = \sfc_{i,\l}^-$, since for this choice of $\beta_{i,\l}$ we have $\sfc_{i,\l}^{-}=(\sfc_{i,\l}^{+})^{-1}$.

\begin{theorem} \label{thm:beta}
There is a 2-isomorphism of graded additive $\Bbbk$-linear categories
\begin{align}
  \cal{U}_Q^{cyc}(\mf{g}) \longrightarrow \cal{U}_{Q,\beta}(\mf{g})
\end{align}
which acts as the identity on objects and 1-morphisms and rescales generating 2-morphisms as follows
\begin{align}
\digamma \left(\;\;
\hackcenter{
\begin{tikzpicture}[scale=0.8]
    \draw[thick, ->] (0,0) -- (0,1.5)
        node[pos=.5, shape=coordinate](DOT){};
    \filldraw  (DOT) circle (2.5pt);
    %\node at (-.85,.85) {\tiny $\lambda +\alpha_i$};
    \node at (.5,.85) {\tiny $\lambda$};
    \node at (-.2,.1) {\tiny $i$};
\end{tikzpicture}}   \right)
\;\;  = \;\;
\hackcenter{
\begin{tikzpicture}[scale=0.8, blue]
    \draw[thick, ->] (0,0) -- (0,1.5)
        node[pos=.5, shape=coordinate](DOT){};
    \filldraw  (DOT) circle (2.5pt);
    %\node at (-.85,.85) {\tiny $\lambda +\alpha_i$};
    \node at (.5,.85) {\tiny $\lambda$};
    \node at (-.2,.1) {\tiny $i$};
\end{tikzpicture}}
   \qquad \qquad
\digamma \left(  \hackcenter{\begin{tikzpicture}[scale=0.8]
    \draw[thick, ->] (0,0) .. controls (0,.5) and (.75,.5) .. (.75,1.0);
    \draw[thick, ->] (.75,0) .. controls (.75,.5) and (0,.5) .. (0,1.0);
    \node at (1.1,.55) {\tiny $\lambda$};
    \node at (-.2,.1) {\tiny $i$};
    \node at (.95,.1) {\tiny $j$};
\end{tikzpicture}} \right)
\;\;   = \;\;
   \hackcenter{\begin{tikzpicture}[scale=0.8, blue]
    \draw[thick, ->] (0,0) .. controls (0,.5) and (.75,.5) .. (.75,1.0);
    \draw[thick, ->] (.75,0) .. controls (.75,.5) and (0,.5) .. (0,1.0);
    \node at (1.1,.55) {\tiny $\lambda$};
    \node at (-.2,.1) {\tiny $i$};
    \node at (.95,.1) {\tiny $j$};
\end{tikzpicture}}
   \nn
\end{align}
\begin{align} \label{eq:digamma-rcap}
    \digamma \left(\;\;
\hackcenter{\begin{tikzpicture}[scale=0.8]
    \draw[thick, <-] (.75,-2) .. controls ++(0,.75) and ++(0,.75) .. (0,-2);
    \node at (.4,-1.2) {\tiny $\lambda$};
    \node at (-.2,-1.9) {\tiny $i$};
\end{tikzpicture}}
\;\;\right)
\;\; & = \;\;
\left\{
  \begin{array}{rl} \frac{c_{i,\l}}{\sfc_{i,\l}^+}
    \hackcenter{\begin{tikzpicture}[scale=0.8, blue]
    \draw[thick, <-] (.75,-2) .. controls ++(0,.75) and ++(0,.75) .. (0,-2);
    \node at (.4,-1.2) {\tiny $\lambda$};
    \node at (-.2,-1.9) {\tiny $i$};
\end{tikzpicture}} & \hbox{ if  $\l_i\equiv_4 0$ or $\l_i\equiv_4 1$ }
\\
% \scs \sfc_{i,\l-\alpha_i}^-
%    \hackcenter{\begin{tikzpicture}[scale=0.8, blue]
%    \draw[thick, <-] (.75,-2) .. controls ++(0,.75) and ++(0,.75) .. (0,-2);
%    \node at (.4,-1.2) {\tiny $\lambda$};
%    \node at (-.2,-1.9) {\tiny $i$};
%\end{tikzpicture}} & \hbox{ if $\l_i \leq 0$ and  $\l_i\equiv_4 0$ or $\l_i\equiv_4 1$ } \\
    \scs c_{i,\l}\;
    \hackcenter{\begin{tikzpicture}[scale=0.8, blue]
    \draw[thick, <-] (.75,-2) .. controls ++(0,.75) and ++(0,.75) .. (0,-2);
    \node at (.4,-1.2) {\tiny $\lambda$};
    \node at (-.2,-1.9) {\tiny $i$};
\end{tikzpicture}} & \hbox{ otherwise.}
  \end{array}
\right.
   \\
 \digamma \left(\;\;
\hackcenter{\begin{tikzpicture}[scale=0.8]
    \draw[thick, ->] (.75,-2) .. controls ++(0,.75) and ++(0,.75) .. (0,-2);
    \node at (.4,-1.2) {\tiny $\lambda$};
    \node at (.95,-1.9) {\tiny $i$};
\end{tikzpicture}}
\;\;\right)
\;\; & = \;\;
\left\{
  \begin{array}{rl}
%   \scs  \sfc_{i,\l+\alpha_i}^+ \;
%   \hackcenter{\begin{tikzpicture}[scale=0.8, blue]
%    \draw[thick, ->] (.75,-2) .. controls ++(0,.75) and ++(0,.75) .. (0,-2);
%    \node at (.4,-1.2) {\tiny $\lambda$};\;
%    \node at (.95,-1.9) {\tiny $i$};
%\end{tikzpicture}} & \hbox{ if $\l_i \geq 0$ and  $\l_i\equiv_4 0$ or $\l_i\equiv_4 1$ } \\
    \frac{1}{\sfc_{i,\l}^-}\;\;
    \hackcenter{\begin{tikzpicture}[scale=0.8, blue]
    \draw[thick, ->] (.75,-2) .. controls ++(0,.75) and ++(0,.75) .. (0,-2);
    \node at (.4,-1.2) {\tiny $\lambda$};
    \node at (.95,-1.9) {\tiny $i$};
\end{tikzpicture}} & \hbox{ if   $\l_i\equiv_4 0$ or $\l_i\equiv_4 1$ } \\
    \hackcenter{\begin{tikzpicture}[scale=0.8, blue]
    \draw[thick, ->] (.75,-2) .. controls ++(0,.75) and ++(0,.75) .. (0,-2);
    \node at (.4,-1.2) {\tiny $\lambda$};
    \node at (.95,-1.9) {\tiny $i$};
\end{tikzpicture}} & \hbox{ otherwise.}
  \end{array}
\right.
  \\
\digamma \left(\;\;
\hackcenter{\begin{tikzpicture}[scale=0.8]
    \draw[thick, ->] (.75,2) .. controls ++(0,-.75) and ++(0,-.75) .. (0,2);
    \node at (.4,1.2) {\tiny $\lambda$};
    \node at (.95,1.9) {\tiny $i$};
\end{tikzpicture}}
\;\;\right)
\;\; & = \;\;
\left\{
  \begin{array}{rl}
    \frac{1}{\sfc_{i,\l}^+}\;\;
    \hackcenter{\begin{tikzpicture}[scale=0.8, blue]
    \draw[thick, ->] (.75,2) .. controls ++(0,-.75) and ++(0,-.75) .. (0,2);
    \node at (.4,1.2) {\tiny $\lambda$};
    \node at (.95,1.9) {\tiny $i$};
\end{tikzpicture}}   & \hbox{if    $\l_i\equiv_4 2$ or $\l_i\equiv_4 3$} \\
%    \scs \sfc_{i,\l-\alpha_i}^-\;\;
%    \hackcenter{\begin{tikzpicture}[scale=0.8, blue]
%    \draw[thick, ->] (.75,2) .. controls ++(0,-.75) and ++(0,-.75) .. (0,2);
%    \node at (.4,1.2) {\tiny $\lambda$};
%    \node at (.95,1.9) {\tiny $i$};
%\end{tikzpicture}}    & \hbox{  if $\l_i \leq 0$ and  $\l_i\equiv_4 0$ or $\l_i\equiv_4 1$} \\
    \hackcenter{\begin{tikzpicture}[scale=0.8, blue]
    \draw[thick, ->] (.75,2) .. controls ++(0,-.75) and ++(0,-.75) .. (0,2);
    \node at (.4,1.2) {\tiny $\lambda$};
    \node at (.95,1.9) {\tiny $i$};
\end{tikzpicture}}   & \hbox{  otherwise}
  \end{array}
\right.
  \\
\digamma \left(
\;\; \hackcenter{\begin{tikzpicture}[scale=0.8]
    \draw[thick, <-] (.75,2) .. controls ++(0,-.75) and ++(0,-.75) .. (0,2);
    \node at (.4,1.2) {\tiny $\lambda$};
    \node at (-.2,1.9) {\tiny $i$};
\end{tikzpicture}}
\;\; \right)
\;\; & = \;\;
\left\{
  \begin{array}{rl}
%    \scs \sfc_{i,\l+\alpha_i}^+
%    \hackcenter{\begin{tikzpicture}[scale=0.8, blue]
%    \draw[thick, <-] (.75,2) .. controls ++(0,-.75) and ++(0,-.75) .. (0,2);
%    \node at (.4,1.2) {\tiny $\lambda$};
%    \node at (-.2,1.9) {\tiny $i$};
%\end{tikzpicture}} & \hbox{if $\l_i \geq 0$ and  $\l_i\equiv_4 2$ or $\l_i\equiv_4 3$} \\
    \frac{c_{i,\l}^{-1}}{\sfc_{i,\l}^-}
    \hackcenter{\begin{tikzpicture}[scale=0.8, blue]
    \draw[thick, <-] (.75,2) .. controls ++(0,-.75) and ++(0,-.75) .. (0,2);
    \node at (.4,1.2) {\tiny $\lambda$};
    \node at (-.2,1.9) {\tiny $i$};
\end{tikzpicture}} & \hbox{if  $\l_i\equiv_4 2$ or $\l_i\equiv_4 3$} \\
    \scs c_{i,\l}^{-1} \;
    \hackcenter{\begin{tikzpicture}[scale=0.8, blue]
    \draw[thick, <-] (.75,2) .. controls ++(0,-.75) and ++(0,-.75) .. (0,2);
    \node at (.4,1.2) {\tiny $\lambda$};
    \node at (-.2,1.9) {\tiny $i$};
\end{tikzpicture}} & \hbox{ otherwise}
  \end{array}
\right.
\nn
\end{align}
\end{theorem}

\begin{proof}
This assignment preserves the scalars $Q$, however by definition we have $t_{ii}^{cyc} = -\beta_i^{cyc} = +1$ in $\cal{U}_Q^{cyc}(\mf{g})$ and $t_{ii}=-\beta_i$ in $\cal{U}_{Q,\beta}(\mf{g})$.  Since $\digamma$ does not effect upward oriented dots and crossings this rescaling has no effect on the KLR-relations.

To see that these definitions preserve the adjunction axioms \eqref{eq_biadjoint1-cyc} and \eqref{eq_biadjoint2-cyc} for caps and cups it is helpful to note that by the properties of compatible bubble coefficients we have
\begin{equation} \label{eq:sfc-rel}
\sfc_{i,\l-\alpha_i}^- \refequal{\eqref{eq:c-change-sl}} t_{ii}\sfc_{i,\l}^-= -\beta_{i,\l} \sfc_{i,\l}^-
\refequal{\eqref{eq:ccinv}} -\beta_{i,\l}\left( \frac{1}{-\beta_{i,\l} \sfc_{i,\l}^+} \right) = \frac{1}{\sfc_{i,\l}^+}.
\end{equation}
The preservation of the dot cyclicity relation \eqref{eq_cyclic_dot-cyc} follows.

A careful case by case analysis shows that crossing cyclicity \eqref{eq_cyclic} is preserved.  This verification produces rescalings that depend of the values of $\l_i$ and $\l_j$. For example, if $\l_i \equiv_4 0$ and $\l_j\equiv_4 0$ there is no rescaling factor on either side of \eqref{eq_cyclic}.  However, for $\l_i\equiv_4 2$ and $\l_j\equiv_4 1$ both sides are rescaled by a factor of $t_{ji}t_{ij}^{-1}/\sfc_{j,\l}^+$.

To verify the remaining relations it is helpful to compute the image of some important composite morphisms.
By definition, the images of real bubbles are given by
\begin{align}
\digamma \left(\;\;
\hackcenter{ \begin{tikzpicture} [scale=.8]
    % Bubble
 \draw  (-.75,1) arc (360:180:.45cm) [thick];
 \draw[<-](-.75,1) arc (0:180:.45cm) [thick];
     \filldraw  [black] (-1.55,.75) circle (2.5pt);
        \node at (-1.3,.3) { $\scriptstyle \ast + r$};
        \node at (-1.4,1.7) { $i $};
 %% L
 \node at (-.2,1.5) { $\lambda $};
\end{tikzpicture}}
\;\; \right)
\;\; = \;\; \frac{c_{i,\l}}{\sfc_{i,\l}^+}\;
\hackcenter{ \begin{tikzpicture} [scale=.8, blue]
    % Bubble
 \draw  (-.75,1) arc (360:180:.45cm) [thick];
 \draw[<-](-.75,1) arc (0:180:.45cm) [thick];
     \filldraw  [black] (-1.55,.75) circle (2.5pt);
        \node at (-1.3,.3) { $\scriptstyle \ast + r$};
        \node at (-1.4,1.7) { $i $};
 %% L
 \node at (-.2,1.5) { $\lambda $};
\end{tikzpicture}}
\qquad \qquad
\digamma \left(\;\;
\hackcenter{ \begin{tikzpicture} [scale=.8]
    % Bubble
 \draw  (-.75,1) arc (360:180:.45cm) [thick];
 \draw[->](-.75,1) arc (0:180:.45cm) [thick];
     \filldraw  [black] (-1.55,.75) circle (2.5pt);
        \node at (-1.3,.3) { $\scriptstyle \ast + r$};
        \node at (-1.4,1.7) { $i $};
 %% L
 \node at (-.2,1.5) { $\lambda $};
\end{tikzpicture}}
\;\; \right)
\;\; = \;\; \frac{c_{i,\l}^{-1}}{\sfc_{i,\l}^-}\;
\hackcenter{ \begin{tikzpicture} [scale=.8, blue]
    % Bubble
 \draw  (-.75,1) arc (360:180:.45cm) [thick];
 \draw[->](-.75,1) arc (0:180:.45cm) [thick];
     \filldraw  [black] (-1.55,.75) circle (2.5pt);
        \node at (-1.3,.3) { $\scriptstyle \ast + r$};
        \node at (-1.4,1.7) { $i $};
 %% L
 \node at (-.2,1.5) { $\lambda $};
\end{tikzpicture}}
\end{align}
so it is clear that the degree zero bubble relations are preserved (a clockwise bubble is equal to $c_{i,\l}$ in $\cal{U}_Q^{cyc}(\mf{g})$ and $\sfc_{i,\l}^+$ in $\cal{U}_{Q,\beta}(\mf{g})$).  For fake bubbles it follows by induction using \eqref{eq:fake1} and \eqref{eq:fake2} that 
\begin{align}
\digamma \left(\;\;
\hackcenter{ \begin{tikzpicture} [scale=.8]
    % Bubble
 \draw  (-.75,1) arc (360:180:.45cm) [thick];
 \draw[<-](-.75,1) arc (0:180:.45cm) [thick];
     \filldraw  [black] (-1.55,.75) circle (2.5pt);
        \node at (-1.3,.3) { $\scriptstyle \ast + r$};
        \node at (-1.4,1.7) { $i $};
 %% L
 \node at (-.2,1.5) { $\lambda $};
\end{tikzpicture}}
\;\; \right)
\;\; = \;\; -\beta_{i,\l} \frac{\sfc_{i,\l}^-}{c_{i,\l}^{-1}}\;
\hackcenter{ \begin{tikzpicture} [scale=.8, blue]
    % Bubble
 \draw  (-.75,1) arc (360:180:.45cm) [thick];
 \draw[<-](-.75,1) arc (0:180:.45cm) [thick];
     \filldraw  [black] (-1.55,.75) circle (2.5pt);
        \node at (-1.3,.3) { $\scriptstyle \ast + r$};
        \node at (-1.4,1.7) { $i $};
 %% L
 \node at (-.2,1.5) { $\lambda $};
\end{tikzpicture}}
\qquad \qquad
\digamma \left(\;\;
\hackcenter{ \begin{tikzpicture} [scale=.8]
    % Bubble
 \draw  (-.75,1) arc (360:180:.45cm) [thick];
 \draw[->](-.75,1) arc (0:180:.45cm) [thick];
     \filldraw  [black] (-1.55,.75) circle (2.5pt);
        \node at (-1.3,.3) { $\scriptstyle \ast + r$};
        \node at (-1.4,1.7) { $i $};
 %% L
 \node at (-.2,1.5) { $\lambda $};
\end{tikzpicture}}
\;\; \right)
\;\; = \;\; -\beta_{i,\l}\frac{\sfc_{i,\l}^+}{c_{i,\l}}\;
\hackcenter{ \begin{tikzpicture} [scale=.8, blue]
    % Bubble
 \draw  (-.75,1) arc (360:180:.45cm) [thick];
 \draw[->](-.75,1) arc (0:180:.45cm) [thick];
     \filldraw  [black] (-1.55,.75) circle (2.5pt);
        \node at (-1.3,.3) { $\scriptstyle \ast + r$};
        \node at (-1.4,1.7) { $i $};
 %% L
 \node at (-.2,1.5) { $\lambda $};
\end{tikzpicture}}
\end{align}
Hence, the right most terms in the $EF$-relations \eqref{eq:modEF} rescale as the product of the rescalings for the real bubble together with the oppositely oriented fake bubble.  In both cases this product is $-\beta_{i,\l}$.

One can also show that
\begin{align}
\digamma \left(\;\;
 \hackcenter{\begin{tikzpicture}[scale=0.8]
    \draw[thick,<-] (0,0) .. controls ++(0,.5) and ++(0,-.5) .. (.75,1);
    \draw[thick] (.75,0) .. controls ++(0,.5) and ++(0,-.5) .. (0,1);
    \draw[thick, ->] (0,1 ) .. controls ++(0,.5) and ++(0,-.5) .. (.75,2);
    \draw[thick] (.75,1) .. controls ++(0,.5) and ++(0,-.5) .. (0,2);
        \node at (-.2,.15) {\tiny $i$};
    \node at (.95,.15) {\tiny $j$};
    \node at (1,1.55) {  $\lambda$};
\end{tikzpicture}} \;\; \right)
\;\; = \;\;
\digamma \left( \;\;
\hackcenter{\begin{tikzpicture}[xscale=-1.0, scale=0.7]
    \draw[thick, ->] (0,0) .. controls (0,.5) and (.75,.5) .. (.75,1.0);
    \draw[thick, ->] (.75,-.5) to (.75,0) .. controls (.75,.5) and (0,.5) .. (0,1.0) to (0,1.5);
    \draw[thick] (0,0) .. controls ++(0,-.4) and ++(0,-.4) .. (-.75,0) to (-.75,1.5);
    %\draw[thick] (.75,0) .. controls ++(0,-1.2) and ++(0,-1.2) .. (-1.5,0) to (-1.55,2);
    \draw[thick, ->] (.75,1.0) .. controls ++(0,.4) and ++(0,.4) .. (1.5,1.0) to (1.5,-.5);
    %\draw[thick, ->] (0,1.0) .. controls ++(0,1.2) and ++(0,1.2) .. (2.25,1.0) to (2.25,-1);
    \node at (-1.3,.55) {  $\lambda$};
    \node at (1.75,-.2) {\tiny $i$};
    \node at (.55,-.2) {\tiny $j$};
    \node at (-.9,1.2) {\tiny $i$};
    \node at (.25,1.2) {\tiny $j$};
\node at (.56,2.5) {
    \begin{tikzpicture}[xscale=1, scale=0.7]
    \draw[thick, ->] (0,0) .. controls (0,.5) and (.75,.5) .. (.75,1.0);
    \draw[thick, ->] (.75,-.5) to (.75,0) .. controls (.75,.5) and (0,.5) .. (0,1.0) to (0,1.5);
    \draw[thick,<-] (0,0) .. controls ++(0,-.4) and ++(0,-.4) .. (-.75,0) to (-.75,1.5);
    %\draw[thick] (.75,0) .. controls ++(0,-1.2) and ++(0,-1.2) .. (-1.5,0) to (-1.55,2);
    \draw[thick] (.75,1.0) .. controls ++(0,.4) and ++(0,.4) .. (1.5,1.0) to (1.5,-.5);
%    \node at (1.75,-.2) {\tiny $i$};
%    \node at (1,-.2) {\tiny $j$};
    \node at (-.9,1.2) {\tiny $i$};
    \node at (.25,1.2) {\tiny $j$};
    \end{tikzpicture}
    };
\end{tikzpicture}}
\;\; \right)
\;\; = \;\;  c_{i,\l}\cdot c_{i,\l+\alpha_j-\alpha_i}^{-1} \frac{1}{\sfc_{i,\l+\alpha_j -\alpha_i}^-} \frac{1}{\sfc_{i,\l}^+}
 \hackcenter{\begin{tikzpicture}[scale=0.8, blue]
    \draw[thick,<-] (0,0) .. controls ++(0,.5) and ++(0,-.5) .. (.75,1);
    \draw[thick] (.75,0) .. controls ++(0,.5) and ++(0,-.5) .. (0,1);
    \draw[thick, ->] (0,1 ) .. controls ++(0,.5) and ++(0,-.5) .. (.75,2);
    \draw[thick] (.75,1) .. controls ++(0,.5) and ++(0,-.5) .. (0,2);
     \node at (1,1.55) {  $\lambda$};
        \node at (-.2,.15) {\tiny $i$};
    \node at (.95,.15) {\tiny $j$};
\end{tikzpicture}}
\\
\digamma \left(\;\;
 \hackcenter{\begin{tikzpicture}[scale=0.8]
    \draw[thick] (0,0) .. controls ++(0,.5) and ++(0,-.5) .. (.75,1);
    \draw[thick, <-] (.75,0) .. controls ++(0,.5) and ++(0,-.5) .. (0,1);
    \draw[thick] (0,1 ) .. controls ++(0,.5) and ++(0,-.5) .. (.75,2);
    \draw[thick, ->] (.75,1) .. controls ++(0,.5) and ++(0,-.5) .. (0,2);
        \node at (-.2,.15) {\tiny $j$};
    \node at (.95,.15) {\tiny $i$};
     \node at (1,1.55) {  $\lambda$};
\end{tikzpicture}} \;\; \right)
\;\; = \;\;
\digamma \left( \;\;
\hackcenter{\begin{tikzpicture}[scale=0.7]
    \draw[thick, ->] (0,0) .. controls (0,.5) and (.75,.5) .. (.75,1.0);
    \draw[thick, ->] (.75,-.5) to (.75,0) .. controls (.75,.5) and (0,.5) .. (0,1.0) to (0,1.5);
    \draw[thick] (0,0) .. controls ++(0,-.4) and ++(0,-.4) .. (-.75,0) to (-.75,1.5);
    %\draw[thick] (.75,0) .. controls ++(0,-1.2) and ++(0,-1.2) .. (-1.5,0) to (-1.55,2);
    \draw[thick, ->] (.75,1.0) .. controls ++(0,.4) and ++(0,.4) .. (1.5,1.0) to (1.5,-.5);
    %\draw[thick, ->] (0,1.0) .. controls ++(0,1.2) and ++(0,1.2) .. (2.25,1.0) to (2.25,-1);
    \node at (1.85,.55) {  $\lambda$};
    \node at (1.75,-.2) {\tiny $i$};
    \node at (.55,-.2) {\tiny $j$};
    \node at (-.9,1.2) {\tiny $i$};
    \node at (.25,1.2) {\tiny $j$};
\node at (.56,2.5) {
    \begin{tikzpicture}[xscale=-1, scale=0.7]
    \draw[thick, ->] (0,0) .. controls (0,.5) and (.75,.5) .. (.75,1.0);
    \draw[thick, ->] (.75,-.5) to (.75,0) .. controls (.75,.5) and (0,.5) .. (0,1.0) to (0,1.5);
    \draw[thick,<-] (0,0) .. controls ++(0,-.4) and ++(0,-.4) .. (-.75,0) to (-.75,1.5);
    %\draw[thick] (.75,0) .. controls ++(0,-1.2) and ++(0,-1.2) .. (-1.5,0) to (-1.55,2);
    \draw[thick] (.75,1.0) .. controls ++(0,.4) and ++(0,.4) .. (1.5,1.0) to (1.5,-.5);
%    \node at (1.75,-.2) {\tiny $i$};
%    \node at (1,-.2) {\tiny $j$};
    \node at (-.9,1.2) {\tiny $i$};
    \node at (.25,1.2) {\tiny $j$};
    \end{tikzpicture}
    };
\end{tikzpicture}}
\;\; \right)
\;\; = \;\;   c_{i,\l}\cdot c_{i,\l+\alpha_j-\alpha_i}^{-1} \frac{1}{\sfc_{i,\l+\alpha_j -\alpha_i}^-} \frac{1}{\sfc_{i,\l}^+}
 \hackcenter{\begin{tikzpicture}[scale=0.8, blue]
    \draw[thick] (0,0) .. controls ++(0,.5) and ++(0,-.5) .. (.75,1);
    \draw[thick, <-] (.75,0) .. controls ++(0,.5) and ++(0,-.5) .. (0,1);
    \draw[thick] (0,1 ) .. controls ++(0,.5) and ++(0,-.5) .. (.75,2);
    \draw[thick, ->] (.75,1) .. controls ++(0,.5) and ++(0,-.5) .. (0,2);
        \node at (-.2,.15) {\tiny $j$};
    \node at (.95,.15) {\tiny $i$};
     \node at (1,1.55) {  $\lambda$};
\end{tikzpicture}}
\end{align}
where the contribution of the bubble parameters from $\cal{U}_{Q,\beta}(\mf{g})$ is $\frac{1}{\sfc_{i,\l+\alpha_j -\alpha_i}^-} \frac{1}{\sfc_{i,\l}^+}\refequal{\eqref{eq:sfc-rel}} \frac{\sfc_{i,\l+\alpha_j}^+}{\sfc_{i,\l}^+} =t_{ij}$ while the contribution from $\cal{U}_{Q}^{cyc}(\mf{g})$ is
\[
c_{i,\l}\cdot c_{i,\l+\alpha_j-\alpha_i}^{-1} =
 (t_{ij}^{cyc})^{-1}.
\]
In particular, when $i \neq j$ the contributions cancel to give a rescaling coefficient of $1$ and when $i=j$ only the $t_{ii} = -\beta_{i,\l}$ contributes.    Hence, the map $\digamma$ preserves the mixed relation \eqref{mixed_rel-cyc} and sends the usual $\mf{sl}_2$ relation to the $\beta_{i,\l}$ modified form in \eqref{eq:modEF}.
\end{proof}

\subsection{Rescaling nilHecke subalgebras} \label{sec:nilhecke}

The nilHecke algebra remains invariant if we rescale a dot by an arbitrary invertible parameter along with the crossing by the inverse of this parameter. Here we record how the 2-category $\cal{U}_{Q,\beta}(\mf{g})$ is effected by this transformation.

\begin{proposition} \label{prop:nilparam}
Given bubble parameters $\beta$ with compatible choice of scalars $Q$, and $D_i \in \Bbbk^{\times}$ for $i \in I$,  define bubble coefficients $\hat{\beta}$ with
\begin{equation} \label{eq:nil-bub}
\hat{\sfc}_{i,\l}^+ := D_i^{-\l_i+1} \sfc_{i,\l}^+ ,
\qquad
\hat{\sfc}_{i,\l}^- := D_i^{\l_i+1}\sfc_{i,\l}^-,
\qquad
 \hat{\beta}_{i,\l} := D_{i}^{-2}\beta_{i,\l}.
\end{equation}
and scalars $\hat{Q}$ with
\begin{equation}
\begin{split}
\hat{t}_{ii} &= -\hat{\beta}_{i,\l} = - D_i^{-2} \beta_{i,\l} = D_i^{-2} t_{ii}
\\
\hat{t}_{ij} &= D_i t_{ij} \qquad \text{if $(\alpha_i, \alpha_j) =-1$, }
\\
\hat{t}_{ik} &=   t_{ik} \qquad \text{if $(\alpha_i, \alpha_k) = 0$. }
\end{split}
\end{equation}
Then $\hat{Q}$ and $\hat{\beta}$ define bubble parameters with a compatible choice of scalars.
\end{proposition}

\begin{proof}
One can check that with these assignments we have
\begin{equation}
 \hat{\sfc}_{i,\l}^+ \hat{\sfc}_{i,\l}^- = - \frac{1}{\hat{\beta}_i}
\end{equation}
so that \eqref{eq:ccinv} is satisfied with the new bubble coefficients.
We also have
\begin{align}
 \hat{\sfc}_{i,\l+\alpha_i}^+
&= D_i^{-\l_i-1} \sfc_{i,\l+\alpha_i}^+
= t_{ii} D_i^{-\l_i-1} \sfc_{i,\l}^+
= t_{ii}  D_i^{-2}\; \hat{\sfc}_{i,\l}^+
= \hat{t}_{ii} \hat{\sfc}_{i,\l}^+
\nn
\\
\hat{\sfc}_{i,\l+\alpha_j}^+
&= D_i^{-\l_i+2} \sfc_{i,\l+\alpha_j}^+
= D_i^{-\l_i+2} t_{ij}  \sfc_{i,\l}^+
= D_i t_{ij}  \hat{\sfc}_{i,\l}^+ \qquad \quad \text{if $(\alpha_i, \alpha_j) =-1$}
\nn
\\
\hat{\sfc}_{i,\l+\alpha_k}^+
&= D_i^{-\l_i+1} \sfc_{i,\l+\alpha_k}^+
= D_i^{-\l_i+1} t_{ik}  \sfc_{i,\l}^+
=   t_{ik}  \hat{\sfc}_{i,\l}^+ =   \hat{t}_{ik}  \hat{\sfc}_{i,\l}^+ \qquad \text{if $(\alpha_i, \alpha_k) =0$,} \nn
\end{align}
so that \eqref{eq:c-change-sl} is also preserved. Equations \eqref{eq:c-change-sl} for $\sfc^-$ follows similarly.
\end{proof}

\begin{proposition}
For each $i \in I$ let $D_i \in \Bbbk$.  Then there is an isomorphism of 2-categories
\begin{equation}
 \beth \maps \cal{U}_{Q,\beta}(\mf{g}) \longrightarrow
\cal{U}_{\hat{Q},\hat{\beta}}(\mf{g})
\end{equation}
defined by
\begin{align}
 \beth \left(\;\;
\hackcenter{
\begin{tikzpicture}[scale=0.8]
    \draw[thick, ->] (0,0) -- (0,1.5)
        node[pos=.5, shape=coordinate](DOT){};
    \filldraw  (DOT) circle (2.5pt);
    %\node at (-.85,.85) {\tiny $\lambda +\alpha_i$};
    \node at (.5,.85) {\tiny $\lambda$};
    \node at (-.2,.1) {\tiny $i$};
\end{tikzpicture}}   \right)
\;\;  &= \;\;
   D_i  \;
\hackcenter{
\begin{tikzpicture}[scale=0.8, blue]
    \draw[thick, ->] (0,0) -- (0,1.5)
        node[pos=.5, shape=coordinate](DOT){};
    \filldraw  (DOT) circle (2.5pt);
    %\node at (-.85,.85) {\tiny $\lambda +\alpha_i$};
    \node at (.5,.85) {\tiny $\lambda$};
    \node at (-.2,.1) {\tiny $i$};
\end{tikzpicture}}
\qquad
\beth \left(  \hackcenter{\begin{tikzpicture}[scale=0.8]
    \draw[thick, ->] (0,0) .. controls (0,.5) and (.75,.5) .. (.75,1.0);
    \draw[thick, ->] (.75,0) .. controls (.75,.5) and (0,.5) .. (0,1.0);
    \node at (1.1,.55) {\tiny $\lambda$};
    \node at (-.2,.1) {\tiny $i$};
    \node at (.95,.1) {\tiny $i$};
\end{tikzpicture}} \right)
\;\;   = \;\;
  D_i^{-1}\;
    \hackcenter{\begin{tikzpicture}[scale=0.8, blue]
    \draw[thick, ->] (0,0) .. controls (0,.5) and (.75,.5) .. (.75,1.0);
    \draw[thick, ->] (.75,0) .. controls (.75,.5) and (0,.5) .. (0,1.0);
    \node at (1.1,.55) {\tiny $\lambda$};
    \node at (-.2,.1) {\tiny $i$};
    \node at (.95,.1) {\tiny $i$};
\end{tikzpicture}}
\end{align}
where the coefficients $\hat{Q}$ and $ \hat{\beta}$ are defined in \eqref{prop:nilparam}.
\end{proposition}

\begin{proof}
The rescaling of the $i$-labelled dot and the $ii$-crossing are necessarily inverse because of the nilHecke dot sliding relation \eqref{eq:nil-dot}.  Since caps and cups are untouched, it is immediate that adjunction axioms and cyclicity are preserved.  For this rescaling of the $ii$-crossing it is clear by \eqref{eq_r2_ij-gen-cyc} that the coefficients $t_{ij}$ must map to $D_i^{-1} \hat{t}_{ij}$ and the other KLR relations follow similarly.  Likewise, the definition of bubble parameters \eqref{eq:degreezero} determines $\hat{\sfc}_{i,\l}^+$ and   $\hat{\sfc}_{i,\l}^-$.  The only relation that requires care is the $\mf{sl}_2$-relations \eqref{eq:modEF}.  For example, the first sl2 relation follows since for $\l_i \geq 0$ the degree $2j$ fake bubble is given by
\begin{align}
 \beth\left(\; \hackcenter{ \begin{tikzpicture} [scale=.7]
 \draw (-.15,.35) node { $\scs i$};
 \draw  (0,0) arc (180:360:0.5cm) [thick];
 \draw[,->](1,0) arc (0:180:0.5cm) [thick];
\filldraw  [black] (.1,-.25) circle (2.5pt);
 \node at (-.15,-.75) {\tiny $-\l_i -1+j$};
 \node at (1.15,1) { $\lambda  $};
\end{tikzpicture} } \; \right)
\;\; & := \;\;
D_i^{-\l_i-1+j}
\hackcenter{ \begin{tikzpicture} [scale=.7]
 \draw (-.15,.35) node { $\scs i$};
 \draw[thick, blue]  (0,0) arc (180:360:0.5cm);
 \draw[thick,blue ,->](1,0) arc (0:180:0.5cm);
\filldraw  [black] (.1,-.25) circle (2.5pt);
 \node at (-.15,-.75) {\tiny $-\l_i -1+j$};
 \node at (1.15,1) { $\lambda  $};
\end{tikzpicture} }
\end{align}
which follows by induction  since
\begin{align}
 \beth\left(\; \hackcenter{ \begin{tikzpicture} [scale=.7]
 \draw (-.15,.35) node { $\scs i$};
 \draw  (0,0) arc (180:360:0.5cm) [thick];
 \draw[,->](1,0) arc (0:180:0.5cm) [thick];
\filldraw  [black] (.1,-.25) circle (2.5pt);
 \node at (.15,-.75) {\tiny $-\l_i -1+j$};
 \node at (1.15,1) { $\lambda  $};
\end{tikzpicture} } \; \right)
\; & := \;
    - \frac{1}{  \sfc_{i,\l}^{+}}  \;\;
\beth \left(\;
\displaystyle \sum_{\stackrel{\scs x+y=j}{\scs x\geq 1}}
 \hackcenter{ \begin{tikzpicture}[scale=.7]
 \draw (-.15,.35) node { $\scs i$};
 \draw  (0,0) arc (180:360:0.5cm) [thick];
 \draw[,<-](1,0) arc (0:180:0.5cm) [thick];
\filldraw  [black] (.1,-.25) circle (2.5pt);
 \node at (.25,-.75) {\tiny $ \l_i-1 +x$};
 \node at (.85,1) { $\lambda$};
\end{tikzpicture} \;\;
\begin{tikzpicture}[scale=.7]
 \draw (-.15,.35) node { $\scs i$};
 \draw  (0,0) arc (180:360:0.5cm) [thick];
 \draw[->](1,0) arc (0:180:0.5cm) [thick];
\filldraw  [black] (.9,-.25) circle (2.5pt);
 \node at (.45,-.75) {\tiny $-\l_i-1+y$};
 \node at (1.15,1.1) { $\;$};
\end{tikzpicture}  }\; \right) \nn
%\\
%\; \; & =\;\;
%    - \frac{1}{D_i^2 \hat{\beta}_{i,\l} D_i^{-\l_i-1}\hat{\sfc}_{i,\l}^{-}}  \; \beth \left(\;
%\displaystyle \sum_{\stackrel{\scs x+y=j}{\scs x\geq 1}}
% \hackcenter{ \begin{tikzpicture}[scale=.7]
% \draw (-.15,.35) node { $\scs i$};
% \draw  (0,0) arc (180:360:0.5cm) [thick];
% \draw[,<-](1,0) arc (0:180:0.5cm) [thick];
%\filldraw  [black] (.1,-.25) circle (2.5pt);
% \node at (.25,-.75) {\tiny $ \l_i-1 +x$};
% \node at (.85,1) { $\lambda$};
%\end{tikzpicture} \;\;
%\begin{tikzpicture}[scale=.7]
% \draw (-.15,.35) node { $\scs i$};
% \draw  (0,0) arc (180:360:0.5cm) [thick];
% \draw[->](1,0) arc (0:180:0.5cm) [thick];
%\filldraw  [black] (.9,-.25) circle (2.5pt);
% \node at (.45,-.75) {\tiny $-\l_i-1+y$};
% \node at (1.15,1.1) { $\;$};
%\end{tikzpicture}  }\; \right) \nn
\\
& \;  =\;
    - \frac{D_i^{-\l_i+1}  }{ \hat{\sfc}_{i,\l}^{+}}  \; \;\beth \left(\;
\displaystyle \sum_{\stackrel{\scs x+y=j}{\scs x\geq 1}}
 \hackcenter{ \begin{tikzpicture}[scale=.7]
 \draw (-.15,.35) node { $\scs i$};
 \draw  (0,0) arc (180:360:0.5cm) [thick];
 \draw[,<-](1,0) arc (0:180:0.5cm) [thick];
\filldraw  [black] (.1,-.25) circle (2.5pt);
 \node at (.25,-.75) {\tiny $ \l_i-1 +x$};
 \node at (.85,1) { $\lambda$};
\end{tikzpicture} \;\;
\begin{tikzpicture}[scale=.7]
 \draw (-.15,.35) node { $\scs i$};
 \draw  (0,0) arc (180:360:0.5cm) [thick];
 \draw[->](1,0) arc (0:180:0.5cm) [thick];
\filldraw  [black] (.9,-.25) circle (2.5pt);
 \node at (.45,-.75) {\tiny $-\l_i-1+y$};
 \node at (1.15,1.1) { $\;$};
\end{tikzpicture}  }\; \right) \nn
\\
& \;  =\;
    - \frac{D_i^{-\l_i+1}  }{  \hat{\sfc}_{i,\l}^{+}}  \;D_i^{\l_i-1+x} D_i^{-\l_i-1+y}
\displaystyle \sum_{\stackrel{\scs x+y=j}{\scs x\geq 1}}
 \hackcenter{ \begin{tikzpicture}[scale=.7]
 \draw (-.15,.35) node { $\scs i$};
 \draw  (0,0) arc (180:360:0.5cm) [thick, blue];
 \draw[,<-](1,0) arc (0:180:0.5cm) [thick, blue];
\filldraw  [black] (.1,-.25) circle (2.5pt);
 \node at (.25,-.75) {\tiny $ \l_i-1 +x$};
 \node at (.85,1) { $\lambda$};
\end{tikzpicture} \;\;
\begin{tikzpicture}[scale=.7]
 \draw (-.15,.35) node { $\scs i$};
 \draw  (0,0) arc (180:360:0.5cm) [thick, blue];
 \draw[->](1,0) arc (0:180:0.5cm) [thick,blue];
\filldraw  [black] (.9,-.25) circle (2.5pt);
 \node at (.45,-.75) {\tiny $-\l_i-1+y$};
 \node at (1.15,1.1) { $\;$};
\end{tikzpicture}  }\;  \nn
\\
& \;  =\;
    - \frac{D_i^{-\l_i-1 +j}  }{  \hat{\sfc}_{i,\l}^{+}}
\displaystyle \sum_{\stackrel{\scs x+y=j}{\scs x\geq 1}}
 \hackcenter{ \begin{tikzpicture}[scale=.7]
 \draw (-.15,.35) node { $\scs i$};
 \draw  (0,0) arc (180:360:0.5cm) [thick, blue];
 \draw[,<-](1,0) arc (0:180:0.5cm) [thick, blue];
\filldraw  [black] (.1,-.25) circle (2.5pt);
 \node at (.25,-.75) {\tiny $ \l_i-1 +x$};
 \node at (.85,1) { $\lambda$};
\end{tikzpicture} \;\;
\begin{tikzpicture}[scale=.7]
 \draw (-.15,.35) node { $\scs i$};
 \draw  (0,0) arc (180:360:0.5cm) [thick, blue];
 \draw[->](1,0) arc (0:180:0.5cm) [thick,blue];
\filldraw  [black] (.9,-.25) circle (2.5pt);
 \node at (.45,-.75) {\tiny $-\l_i-1+y$};
 \node at (1.15,1.1) { $\;$};
\end{tikzpicture}  }\;
\; = \;
D_i^{-\l_i-1 +j}  \;
\hackcenter{ \begin{tikzpicture} [scale=.7]
 \draw (-.15,.35) node { $\scs i$};
 \draw  (0,0) arc (180:360:0.5cm) [thick, blue];
 \draw[,->](1,0) arc (0:180:0.5cm) [thick, blue];
\filldraw  [black] (.1,-.25) circle (2.5pt);
 \node at (.15,-.75) {\tiny $\l_i -1+j$};
 \node at (1.15,1) { $\lambda  $};
\end{tikzpicture} } \nn
\end{align}
Hence, the $\mf{sl}_2$ relation follows since
\begin{align}
  \beth \left( \beta_{i,\l}\;
\sum_{\overset{f_1+f_2+f_3}{=\l_i-1}}\hackcenter{
 \begin{tikzpicture}[scale=0.8]
 \draw[thick,->] (0,-1.0) .. controls ++(0,.5) and ++ (0,.5) .. (.8,-1.0) node[pos=.75, shape=coordinate](DOT1){};
  \draw[thick,<-] (0,1.0) .. controls ++(0,-.5) and ++ (0,-.5) .. (.8,1.0) node[pos=.75, shape=coordinate](DOT3){};
 \draw[thick,->] (0,0) .. controls ++(0,-.45) and ++ (0,-.45) .. (.8,0)node[pos=.25, shape=coordinate](DOT2){};
 \draw[thick] (0,0) .. controls ++(0,.45) and ++ (0,.45) .. (.8,0);
 \draw (-.15,.7) node { $\scs i$};
\draw (1.05,0) node { $\scs i$};
\draw (-.15,-.7) node { $\scs i$};
% \draw (.3,.125) node {};
% \draw  (0,0) arc (180:360:0.5cm) [thick];
% \draw[,<-](1,0) arc (0:180:0.5cm) [thick];
%\filldraw  [black] (.1,-.25) circle (2.5pt);
 \node at (.95,.65) {\tiny $f_3$};
 \node at (-.55,-.05) {\tiny $\overset{-\l_i-1}{+f_2}$};
  \node at (.95,-.65) {\tiny $f_1$};
 \node at (1.95,.3) { $\lambda $};
 \filldraw[thick]  (DOT3) circle (2.5pt);
  \filldraw[thick]  (DOT2) circle (2.5pt);
  \filldraw[thick]  (DOT1) circle (2.5pt);
\end{tikzpicture} }  \; \right)
& \; = \;
  \beta_{i,\l}
\sum_{\overset{f_1+f_2+f_3}{=\l_i-1}}
\; D_i^{f_1 + f_3} D_i^{-\l_i -1 + f_2}
\hackcenter{
 \begin{tikzpicture}[scale=0.8]
 \draw[thick,blue, ->] (0,-1.0) .. controls ++(0,.5) and ++ (0,.5) .. (.8,-1.0) node[pos=.75, shape=coordinate](DOT1){};
  \draw[thick,blue,<-] (0,1.0) .. controls ++(0,-.5) and ++ (0,-.5) .. (.8,1.0) node[pos=.75, shape=coordinate](DOT3){};
 \draw[thick,blue,->] (0,0) .. controls ++(0,-.45) and ++ (0,-.45) .. (.8,0)node[pos=.25, shape=coordinate](DOT2){};
 \draw[thick,blue] (0,0) .. controls ++(0,.45) and ++ (0,.45) .. (.8,0);
 \draw (-.15,.7) node { $\scs i$};
\draw (1.05,0) node { $\scs i$};
\draw (-.15,-.7) node { $\scs i$};
% \draw (.3,.125) node {};
% \draw  (0,0) arc (180:360:0.5cm) [thick];
% \draw[,<-](1,0) arc (0:180:0.5cm) [thick];
%\filldraw  [black] (.1,-.25) circle (2.5pt);
 \node at (.95,.65) {\tiny $f_3$};
 \node at (-.55,-.05) {\tiny $\overset{-\l_i-1}{+f_2}$};
  \node at (.95,-.65) {\tiny $f_1$};
 \node at (1.95,.3) { $\lambda $};
 \filldraw[thick,blue]  (DOT3) circle (2.5pt);
  \filldraw[thick,blue]  (DOT2) circle (2.5pt);
  \filldraw[thick,blue]  (DOT1) circle (2.5pt);
\end{tikzpicture} }
\nn
\\
& \; = \;
  \hat{\beta}_{i,\l}
\sum_{\overset{f_1+f_2+f_3}{=\l_i-1}}
\hackcenter{
 \begin{tikzpicture}[scale=0.8]
 \draw[thick,blue, ->] (0,-1.0) .. controls ++(0,.5) and ++ (0,.5) .. (.8,-1.0) node[pos=.75, shape=coordinate](DOT1){};
  \draw[thick,blue,<-] (0,1.0) .. controls ++(0,-.5) and ++ (0,-.5) .. (.8,1.0) node[pos=.75, shape=coordinate](DOT3){};
 \draw[thick,blue,->] (0,0) .. controls ++(0,-.45) and ++ (0,-.45) .. (.8,0)node[pos=.25, shape=coordinate](DOT2){};
 \draw[thick,blue] (0,0) .. controls ++(0,.45) and ++ (0,.45) .. (.8,0);
 \draw (-.15,.7) node { $\scs i$};
\draw (1.05,0) node { $\scs i$};
\draw (-.15,-.7) node { $\scs i$};
% \draw (.3,.125) node {};
% \draw  (0,0) arc (180:360:0.5cm) [thick];
% \draw[,<-](1,0) arc (0:180:0.5cm) [thick];
%\filldraw  [black] (.1,-.25) circle (2.5pt);
 \node at (.95,.65) {\tiny $f_3$};
 \node at (-.55,-.05) {\tiny $\overset{-\l_i-1}{+f_2}$};
  \node at (.95,-.65) {\tiny $f_1$};
 \node at (1.95,.3) { $\lambda $};
 \filldraw[thick,blue]  (DOT3) circle (2.5pt);
  \filldraw[thick,blue]  (DOT2) circle (2.5pt);
  \filldraw[thick,blue]  (DOT1) circle (2.5pt);
\end{tikzpicture} } \nn
\end{align}
and
\begin{equation}
 \beth \left(\;\beta_{i,\l}\;
 \hackcenter{\begin{tikzpicture}[scale=0.8]
    \draw[thick] (0,0) .. controls ++(0,.5) and ++(0,-.5) .. (.75,1);
    \draw[thick, <-] (.75,0) .. controls ++(0,.5) and ++(0,-.5) .. (0,1);
    \draw[thick] (0,1 ) .. controls ++(0,.5) and ++(0,-.5) .. (.75,2);
    \draw[thick, ->] (.75,1) .. controls ++(0,.5) and ++(0,-.5) .. (0,2);
        \node at (-.2,.15) {\tiny $i$};
    \node at (.95,.15) {\tiny $i$};
\end{tikzpicture}} \;\; \right)
\;\; = \;\; D_i^{-2} \beta_{i,\l}
 \hackcenter{\begin{tikzpicture}[scale=0.8]
    \draw[thick, blue] (0,0) .. controls ++(0,.5) and ++(0,-.5) .. (.75,1);
    \draw[thick, blue, <-] (.75,0) .. controls ++(0,.5) and ++(0,-.5) .. (0,1);
    \draw[thick, blue] (0,1 ) .. controls ++(0,.5) and ++(0,-.5) .. (.75,2);
    \draw[thick, blue, ->] (.75,1) .. controls ++(0,.5) and ++(0,-.5) .. (0,2);
        \node at (-.2,.15) {\tiny $i$};
    \node at (.95,.15) {\tiny $i$};
\end{tikzpicture}}
\;\; = \;\; \hat{\beta}_{i,\l}
 \hackcenter{\begin{tikzpicture}[scale=0.8]
    \draw[thick, blue] (0,0) .. controls ++(0,.5) and ++(0,-.5) .. (.75,1);
    \draw[thick, blue, <-] (.75,0) .. controls ++(0,.5) and ++(0,-.5) .. (0,1);
    \draw[thick, blue] (0,1 ) .. controls ++(0,.5) and ++(0,-.5) .. (.75,2);
    \draw[thick, blue, ->] (.75,1) .. controls ++(0,.5) and ++(0,-.5) .. (0,2);
        \node at (-.2,.15) {\tiny $i$};
    \node at (.95,.15) {\tiny $i$};
\end{tikzpicture}} \nn
\end{equation}
The other $\mf{sl}_2$ relation follows similarly.
\end{proof}

% - - - - - - - - - - - - -  - - - - - - - - - - - - -
\subsection{Rescalings for a preferred choice of scalars $Q$ } \label{sec:generalQ}
% - - - - - - - - - - - - -  - - - - - - - - - - - - -

The choice of scalars $Q$ controls the form of the KLR algebra $R_Q$ that governs the upward oriented strands.  The algebras $R_Q$ are governed by the products $v_{ij}=t_{ij}^{-1}t_{ji}$ taken
over all pairs $i,j\in I$.   The products $v_{ij}$ can be thought of as a $\Bbbk^{\times}$-valued 1-cocycle on the graph $\Gamma$ associated to the symmetric Cartan data; we call two choices {\it cohomologous} if these 1-cocycles are in the same cohomology class.
For cohomologous choices of scalars $Q$ and $Q'$ the algebras $R_Q$ and $R_Q'$ are isomorphic, though not canonically.
If $\Gamma$ is a tree,  in particular, a Dynkin diagram, then all choices of scalars are cohomologous.
However, an explicit isomorphism of algebras has not been presented in the literature.  Here we give this isomorphism $R_Q \to R_{Q'}$ for any choices of scalars $Q$ and $Q'$ and discuss how to extend to the level of 2-categories $\cal{U}_{Q,\beta}(\mf{g})$ and $\cal{U}_{Q',\beta}(\mf{g})$.  In \cite{Kash-param} Kashiwara studies the effect of these parameters on simple modules for the KLR algebra and show that these generically correspond to upper global canonical bases.

\subsubsection{KLR isomorphisms for trees}

Let $\Gamma$ be the underlying graph associated to a simply-laced Kac-Moody algebra.  If $\Gamma$ is a tree, then any two vertices $i,j \in I$ of the graph are connected by a unique path.  To specify the isomorphism we designate a distinguished vertex $r \in I$ as the root of the tree.  This defines an orientation on the graph in which all edges are oriented away from $r$.  Define the \emph{level} of a vertex $i \in I$ to be its distance from the root $r$.  We write $i < j$ if vertex $i$ occurs in a lower level than vertex $j$.

Denote by  $P_{i}$ the unique directed path of edges from the root vertex $r$ to vertex $i$.  For any edge $e \in \Gamma$ we write $e \in P_{i}$ if $e$ appears in the path from $r$ to $i$.  Let $s(e) \in I$ denote the source vertex of the edge $e$ and $t(e) \in I$ denote the target of the directed edge $e$.

Define parameters
\begin{equation} \label{eq:Di}
 D_i :=  \displaystyle\prod_{e \in P_{i}} v_{s(e) t(e)}' v_{t(e) s(e)},\qquad
 D_{i}^{-1} := \displaystyle\prod_{e \in P_{i}} v_{ t(e)s(e)}' v_{s(e)t(e) }
\end{equation}
so that $D_i^{-1}$ is the inverse of $D_i$, since $v_{ij} = t_{ij}^{-1} t_{ji}$ so that $v_{ji} = v_{ij}^{-1}$.  For the root of the tree $r \in I$ we have $D_r =1$.

\begin{proposition}
Let $Q$ and $Q'$ be choices of scalars for a KLR algebra associated to a simply-laced Kac-Moody algebra whose underlying graph is a tree.   Then associated to a fixed choice of root for the tree, there is an algebra isomorphism
\begin{align}
  \gimel \maps R_Q \longrightarrow R_{Q'}
\end{align}
defined by
\begin{align} \label{eq:gimel-KLR}
 \gimel \left(\;\;
\hackcenter{
\begin{tikzpicture}[scale=0.8]
    \draw[thick, ->] (0,0) -- (0,1.5)
        node[pos=.5, shape=coordinate](DOT){};
    \filldraw  (DOT) circle (2.5pt);
    %\node at (-.85,.85) {\tiny $\lambda +\alpha_i$};
    \node at (.5,.85) {\tiny $\lambda$};
    \node at (-.2,.1) {\tiny $i$};
\end{tikzpicture}}   \right)
\;\;  &= \;\;
   D_i  \;
\hackcenter{
\begin{tikzpicture}[scale=0.8, blue]
    \draw[thick, ->] (0,0) -- (0,1.5)
        node[pos=.5, shape=coordinate](DOT){};
    \filldraw  (DOT) circle (2.5pt);
    %\node at (-.85,.85) {\tiny $\lambda +\alpha_i$};
    \node at (.5,.85) {\tiny $\lambda$};
    \node at (-.2,.1) {\tiny $i$};
\end{tikzpicture}}
 \\
\gimel \left(  \hackcenter{\begin{tikzpicture}[scale=0.8]
    \draw[thick, ->] (0,0) .. controls (0,.5) and (.75,.5) .. (.75,1.0);
    \draw[thick, ->] (.75,0) .. controls (.75,.5) and (0,.5) .. (0,1.0);
    \node at (1.1,.55) {\tiny $\lambda$};
    \node at (-.2,.1) {\tiny $i$};
    \node at (.95,.1) {\tiny $j$};
\end{tikzpicture}} \right)
\;\;   &= \;\;
\left\{
  \begin{array}{cl}
  D_i^{-1}\;
    \hackcenter{\begin{tikzpicture}[scale=0.8, blue]
    \draw[thick, ->] (0,0) .. controls (0,.5) and (.75,.5) .. (.75,1.0);
    \draw[thick, ->] (.75,0) .. controls (.75,.5) and (0,.5) .. (0,1.0);
    \node at (1.1,.55) {\tiny $\lambda$};
    \node at (-.2,.1) {\tiny $i$};
    \node at (.95,.1) {\tiny $j$};
\end{tikzpicture}}
    & \hbox{if $i = j$,}
 \smallskip\\
  t_{ji}(t_{ji}')^{-1}  D_j  \;
  \hackcenter{\begin{tikzpicture}[scale=0.8, blue]
    \draw[thick, ->] (0,0) .. controls (0,.5) and (.75,.5) .. (.75,1.0);
    \draw[thick, ->] (.75,0) .. controls (.75,.5) and (0,.5) .. (0,1.0);
    \node at (1.1,.55) {\tiny $\lambda$};
    \node at (-.2,.1) {\tiny $i$};
    \node at (.95,.1) {\tiny $j$};
\end{tikzpicture}}
 & \hbox{if $i < j$ and $i \cdot j =-1$,}
\\
 t_{ij}(t_{ij}')^{-1}    \;
  \hackcenter{\begin{tikzpicture}[scale=0.8, blue]
    \draw[thick, ->] (0,0) .. controls (0,.5) and (.75,.5) .. (.75,1.0);
    \draw[thick, ->] (.75,0) .. controls (.75,.5) and (0,.5) .. (0,1.0);
    \node at (1.1,.55) {\tiny $\lambda$};
    \node at (-.2,.1) {\tiny $i$};
    \node at (.95,.1) {\tiny $j$};
\end{tikzpicture}}
 & \hbox{if $i < j$ and $i \cdot j =0$,}
\\
     \hackcenter{\begin{tikzpicture}[scale=0.8, blue]
    \draw[thick, ->] (0,0) .. controls (0,.5) and (.75,.5) .. (.75,1.0);
    \draw[thick, ->] (.75,0) .. controls (.75,.5) and (0,.5) .. (0,1.0);
    \node at (1.1,.55) {\tiny $\lambda$};
    \node at (-.2,.1) {\tiny $i$};
    \node at (.95,.1) {\tiny $j$};
\end{tikzpicture}}
   & \hbox{otherwise,}
  \end{array}
\right.
   \nn
\end{align}
with parameters $D_i$ for $i\in I$ defined in \eqref{eq:Di}.
\end{proposition}

\begin{proof}
It is immediate that the nilHecke relations are preserved because the $i$ colored dot is rescaled inversely to the $ii$ colored crossing.
If $i < j$ and $i\cdot j =-1$, then the path $P_i$ is identical to the path $P_j$ with the addition of the edge from $i \to j$. Hence,
\begin{equation} \label{eq:PiPj}
D_j
= \displaystyle\prod_{e \in P_{j}} v_{s(e) t(e)}' v_{t(e) s(e)}
=   v_{ij}' v_{ji}\displaystyle\prod_{e \in P_{i}} v_{s(e) t(e)}' v_{t(e) s(e)}
= v_{ij}' v_{ji} D_i
\end{equation}
and in particular,
\begin{equation} \label{AA1}
  t_{ji}(t_{ji}')^{-1}  D_j\; =\;
 t_{ji}(t_{ji}')^{-1} v_{ij}' v_{ji} D_i
\; = \;
t_{ij}(t_{ij}')^{-1}  D_i.
\end{equation}
Hence, if $i<j$ and $i\cdot j =-1$ then
\begin{align}
\gimel \left( \; \hackcenter{
\begin{tikzpicture}[scale=0.8]
    \draw[thick, ->] (0,0) .. controls ++(0,.5) and ++(0,-.4) .. (.75,.8) .. controls ++(0,.4) and ++(0,-.5) .. (0,1.6);
    \draw[thick, ->] (.75,0) .. controls ++(0,.5) and ++(0,-.4) .. (0,.8) .. controls ++(0,.4) and ++(0,-.5) .. (.75,1.6);
    \node at (1.1,1.25) { $\lambda$};
    \node at (-.2,.1) {\tiny $i$};
    \node at (.95,.1) {\tiny $j$};
\end{tikzpicture}} \; \right)
\; = \; t_{ji}(t_{ji}')^{-1} D_j
\hackcenter{
\begin{tikzpicture}[scale=0.8]
    \draw[thick, blue, ->] (0,0) .. controls ++(0,.5) and ++(0,-.4) .. (.75,.8) .. controls ++(0,.4) and ++(0,-.5) .. (0,1.6);
    \draw[thick,blue, ->] (.75,0) .. controls ++(0,.5) and ++(0,-.4) .. (0,.8) .. controls ++(0,.4) and ++(0,-.5) .. (.75,1.6);
    \node at (1.1,1.25) { $\lambda$};
    \node at (-.2,.1) {\tiny $i$};
    \node at (.95,.1) {\tiny $j$};
\end{tikzpicture}}
\; = \; t_{ij}(t_{ij}')^{-1} D_i
\hackcenter{\begin{tikzpicture}[scale=0.8]
    \draw[thick, blue, ->] (0,0) .. controls ++(0,.5) and ++(0,-.4) .. (.75,.8) .. controls ++(0,.4) and ++(0,-.5) .. (0,1.6);
    \draw[thick,blue, ->] (.75,0) .. controls ++(0,.5) and ++(0,-.4) .. (0,.8) .. controls ++(0,.4) and ++(0,-.5) .. (.75,1.6);
    \node at (1.1,1.25) { $\lambda$};
    \node at (-.2,.1) {\tiny $i$};
    \node at (.95,.1) {\tiny $j$};
\end{tikzpicture}}
\end{align}
which, together with the dot rescaling, transforms the quadratic KLR relation in parameters $Q$ to the corresponding relation with parameters $Q'$.  A similar argument establishes the quadratic KLR relation for $j<i$ and $i\cdot j=-1$.  The quadratic relation for $i\cdot j =0$ is immediate.

All of the cubic KLR relations follow immediately except for the case \eqref{eq:KLR-r3}.  If $i < j$ and $i\cdot j =-1$ then
\begin{align}
& \nn\gimel
\left( \;
\hackcenter{\begin{tikzpicture}[scale=0.7]
    \draw[thick, ->] (0,0) .. controls ++(0,1) and ++(0,-1) .. (1.5,2);
    \draw[thick, ] (.75,0) .. controls ++(0,.5) and ++(0,-.5) .. (0,1);
    \draw[thick, ->] (0,1) .. controls ++(0,.5) and ++(0,-.5) .. (0.75,2);
    \draw[thick, ->] (1.5,0) .. controls ++(0,1) and ++(0,-1) .. (0,2);
    \node at (-.2,.15) {\tiny $i$};
    \node at (.95,.15) {\tiny $j$};
    \node at (1.75,.15) {\tiny $i$};
\end{tikzpicture}}
\;- \;
\hackcenter{\begin{tikzpicture}[scale=0.7]
    \draw[thick, ->] (0,0) .. controls ++(0,1) and ++(0,-1) .. (1.5,2);
    \draw[thick, ] (.75,0) .. controls ++(0,.5) and ++(0,-.5) .. (1.5,1);
    \draw[thick, ->] (1.5,1) .. controls ++(0,.5) and ++(0,-.5) .. (0.75,2);
    \draw[thick, ->] (1.5,0) .. controls ++(0,1) and ++(0,-1) .. (0,2);
    \node at (-.2,.15) {\tiny $i$};
    \node at (.95,.15) {\tiny $j$};
    \node at (1.75,.15) {\tiny $i$};
\end{tikzpicture}}
\; \right)
\; \refequal{\eqref{AA1}} \;
t_{ij}(t_{ij}')^{-1}
\left(
\;
\hackcenter{\begin{tikzpicture}[scale=0.7]
    \draw[thick, blue,->] (0,0) .. controls ++(0,1) and ++(0,-1) .. (1.5,2);
    \draw[thick,blue, ] (.75,0) .. controls ++(0,.5) and ++(0,-.5) .. (0,1);
    \draw[thick,blue, ->] (0,1) .. controls ++(0,.5) and ++(0,-.5) .. (0.75,2);
    \draw[thick, blue,->] (1.5,0) .. controls ++(0,1) and ++(0,-1) .. (0,2);
    \node at (-.2,.15) {\tiny $i$};
    \node at (.95,.15) {\tiny $j$};
    \node at (1.75,.15) {\tiny $i$};
\end{tikzpicture}}
\;- \;
\hackcenter{\begin{tikzpicture}[scale=0.7]
    \draw[thick,blue, ->] (0,0) .. controls ++(0,1) and ++(0,-1) .. (1.5,2);
    \draw[thick,blue, ] (.75,0) .. controls ++(0,.5) and ++(0,-.5) .. (1.5,1);
    \draw[thick,blue, ->] (1.5,1) .. controls ++(0,.5) and ++(0,-.5) .. (0.75,2);
    \draw[thick,blue, ->] (1.5,0) .. controls ++(0,1) and ++(0,-1) .. (0,2);
    \node at (-.2,.15) {\tiny $i$};
    \node at (.95,.15) {\tiny $j$};
    \node at (1.75,.15) {\tiny $i$};
\end{tikzpicture}}
\right)
\nn
\end{align}
so that the nontrivial cubic KLR relation also follows.
\end{proof}

\subsubsection{KLR rescaling for 2-categories}

The isomorphism $\gimel \maps R_Q \to R_{Q'}$ of KLR algebras can be extended to an isomorphism of 2-categories.    Without rescaling the caps and cups, the assignment \eqref{eq:gimel-KLR} will rescale the  bubble parameters as in \eqref{eq:nil-bub}, but these new parameters need not be compatible with the coefficients $Q'$.  If we assume that a set of bubble parameters $\beta'$ compatible with $Q'$ have been specified, then it is possible to rescale appropriately to extend $\gimel$ to an isomorphism of 2-categories.

\begin{theorem}
 Let $\beta$, respectively $\beta'$, be bubble parameters with compatible choice of scalars $Q$, respectively $Q'$.  Then the algebra isomorphism $\gimel \maps R_Q \to R_{Q'}$ from \eqref{eq:gimel-KLR} extends to an isomorphism of 2-categories
\begin{equation}
 \gimel \maps \; \cal{U}_{Q,\beta} \; \longrightarrow \; \cal{U}_{Q',\beta'}
\end{equation}
given by
\begin{align}
   \beth\left(\;\;
\hackcenter{\begin{tikzpicture}[scale=0.8]
    \draw[thick, <-] (.75,-2) .. controls ++(0,.75) and ++(0,.75) .. (0,-2);
    \node at (.4,-1.2) {\tiny $\lambda$};
    \node at (-.2,-1.9) {\tiny $i$};
\end{tikzpicture}}
\;\;\right)
\;\;  = \;\; D_i^{-\l_i+1} \frac{\sfc_{i,\l}^{ +}}{\sfc_{i,\l}^{\prime +}}
    \hackcenter{\begin{tikzpicture}[scale=0.8, blue]
    \draw[thick, <-] (.75,-2) .. controls ++(0,.75) and ++(0,.75) .. (0,-2);
    \node at (.4,-1.2) {\tiny $\lambda$};
    \node at (-.2,-1.9) {\tiny $i$};
\end{tikzpicture}}
\qquad \qquad
\beth \left(
\;\; \hackcenter{\begin{tikzpicture}[scale=0.8]
    \draw[thick, <-] (.75,2) .. controls ++(0,-.75) and ++(0,-.75) .. (0,2);
    \node at (.4,1.2) {\tiny $\lambda$};
    \node at (-.2,1.9) {\tiny $i$};
\end{tikzpicture}}
\;\; \right)
\;\;  = \;\; D_{i}^{\l_i+1} \frac{\sfc_{i,\l+\alpha_i}^{\prime +}}{\sfc_{i,\l+\alpha_i}^{+}}
    \hackcenter{\begin{tikzpicture}[scale=0.8, blue]
    \draw[thick, <-] (.75,2) .. controls ++(0,-.75) and ++(0,-.75) .. (0,2);
    \node at (.4,1.2) {\tiny $\lambda$};
    \node at (-.2,1.9) {\tiny $i$};
\end{tikzpicture}}
\end{align}
and sending the other cap and cup to themselves.
\end{theorem}

\begin{proof}
 Note that the coefficient for the rightward pointing cup can be rewritten
\[
 D_{i}^{\l_i+1} \frac{\sfc_{i,\l+\alpha_i}^{\prime +}}{\sfc_{i,\l+\alpha_i}^{+}}
\; \refequal{\eqref{eq:c-change-sl}} \;
D_{i}^{\l_i+1}  \frac{t_{ii}'}{t_{ii}} \frac{\sfc_{i,\l}^{\prime +}}{\sfc_{i,\l}^{+}}
\; \refequal{\eqref{eq:ccinv}}\;
D_{i}^{\l_i+1}   \frac{\sfc_{i,\l}^{-}}{\sfc_{i,\l}^{\prime -}}.
\]
Using this it is not difficult to show that the adjunction axioms and dot cyclicity are preserved.  It is also straightforward to show that bubble coefficients are preserved.  For example,
\begin{align}
\sfc_{i,\l}^+ \; = \;
\gimel \left(\;\;
\hackcenter{ \begin{tikzpicture} [scale=.8]
    % Bubble
 \draw  (-.75,1) arc (360:180:.45cm) [thick];
 \draw[<-](-.75,1) arc (0:180:.45cm) [thick];
     \filldraw  [black] (-1.55,.75) circle (2.5pt);
        \node at (-1.3,.3) { $\scriptstyle \l_i-1$};
        \node at (-1.4,1.7) { $i $};
 %% L
 \node at (-.2,1.5) { $\lambda $};
\end{tikzpicture}}
\;\; \right)
\;\; = \;\;
D_i^{\l_i-1} \cdot D_i^{-\l_i+1} \frac{\sfc_{i,\l}^{ +}}{\sfc_{i,\l}^{\prime +}}
 \;\;
\hackcenter{ \begin{tikzpicture} [scale=.8, blue]
    % Bubble
 \draw  (-.75,1) arc (360:180:.45cm) [thick];
 \draw[<-](-.75,1) arc (0:180:.45cm) [thick];
     \filldraw  [black] (-1.55,.75) circle (2.5pt);
        \node at (-1.3,.3) { $\scriptstyle \l_i-1$};
        \node at (-1.4,1.7) { $i $};
 %% L
 \node at (-.2,1.5) { $\lambda $};
\end{tikzpicture}}
\end{align}
For crossing cyclicity let $\gamma$ be the rescaling of the $ij$ crossing (for arbitrary relationship between $i$ and $j$) so that
\begin{align} \nn
& \gimel \left(
\hackcenter{\begin{tikzpicture}[scale=0.7]
    \draw[thick, ->] (0,0) .. controls (0,.5) and (.75,.5) .. (.75,1.0);
    \draw[thick, ->] (.75,0) .. controls (.75,.5) and (0,.5) .. (0,1.0);
    \draw[thick] (0,0) .. controls ++(0,-.4) and ++(0,-.4) .. (-.75,0) to (-.75,2);
    \draw[thick] (.75,0) .. controls ++(0,-1.2) and ++(0,-1.2) .. (-1.5,0) to (-1.55,2);
    \draw[thick, ->] (.75,1.0) .. controls ++(0,.4) and ++(0,.4) .. (1.5,1.0) to (1.5,-1);
    \draw[thick, ->] (0,1.0) .. controls ++(0,1.2) and ++(0,1.2) .. (2.25,1.0) to (2.25,-1);
    \node at (-.35,.75) {  $\lambda$};
    \node at (1.3,-.7) {\tiny $i$};
    \node at (2.05,-.7) {\tiny $j$};
    \node at (-.9,1.7) {\tiny $i$};
    \node at (-1.7,1.7) {\tiny $j$};
\end{tikzpicture}}
\; \right)
 =   \gamma
D_j^{-\l_j+1} \frac{\sfc_{j,\l}^{ +}}{\sfc_{j,\l}^{\prime +}}
\cdot
% lambda = \lambda - \alpha_j so \lambda_i = \lambda_i +1
D_i^{-\l_i} \frac{\sfc_{i,\l-\alpha_j}^{ +}}{\sfc_{i,\l-\alpha_j}^{\prime +}}
\cdot
% \lambda = \lambda - \alpha_i
D_{i}^{\l_i-1} \frac{\sfc_{i,\l}^{\prime +}}{\sfc_{i,\l}^{+}}
\cdot
% lambda = \lambda -\alpha_i -\alpha_j  so \lambda_j -> \lambda_j -2 +1 = \lambda_j -1
D_{j}^{\l_j} \frac{\sfc_{j,\l-\alpha_i}^{\prime +}}{\sfc_{j,\l-\alpha_i}^{+}}
\hackcenter{\begin{tikzpicture}[scale=0.7]
    \draw[thick,blue, ->] (0,0) .. controls (0,.5) and (.75,.5) .. (.75,1.0);
    \draw[thick, blue,->] (.75,0) .. controls (.75,.5) and (0,.5) .. (0,1.0);
    \draw[thick, blue] (0,0) .. controls ++(0,-.4) and ++(0,-.4) .. (-.75,0) to (-.75,2);
    \draw[thick, blue] (.75,0) .. controls ++(0,-1.2) and ++(0,-1.2) .. (-1.5,0) to (-1.55,2);
    \draw[thick, blue, ->] (.75,1.0) .. controls ++(0,.4) and ++(0,.4) .. (1.5,1.0) to (1.5,-1);
    \draw[thick, blue, ->] (0,1.0) .. controls ++(0,1.2) and ++(0,1.2) .. (2.25,1.0) to (2.25,-1);
    \node at (-.35,.75) {  $\lambda$};
    \node at (1.3,-.7) {\tiny $i$};
    \node at (2.05,-.7) {\tiny $j$};
    \node at (-.9,1.7) {\tiny $i$};
    \node at (-1.7,1.7) {\tiny $j$};
\end{tikzpicture}}
\\ \nn
& \qquad  =  \gamma
D_j D_{i}^{-1}
\frac{\sfc_{j,\l-\alpha_i}^{\prime +}}{\sfc_{j,\l}^{\prime +}}
\frac{\sfc_{i,\l}^{\prime +}}{\sfc_{i,\l-\alpha_j}^{\prime +}}
\frac{\sfc_{i,\l-\alpha_j}^{ +}}{\sfc_{i,\l}^{+}}
\frac{\sfc_{j,\l}^{ +}}{\sfc_{j,\l-\alpha_i}^{+}}
\hackcenter{\begin{tikzpicture}[scale=0.7]
    \draw[thick,blue, ->] (0,0) .. controls (0,.5) and (.75,.5) .. (.75,1.0);
    \draw[thick, blue,->] (.75,0) .. controls (.75,.5) and (0,.5) .. (0,1.0);
    \draw[thick, blue] (0,0) .. controls ++(0,-.4) and ++(0,-.4) .. (-.75,0) to (-.75,2);
    \draw[thick, blue] (.75,0) .. controls ++(0,-1.2) and ++(0,-1.2) .. (-1.5,0) to (-1.55,2);
    \draw[thick, blue, ->] (.75,1.0) .. controls ++(0,.4) and ++(0,.4) .. (1.5,1.0) to (1.5,-1);
    \draw[thick, blue, ->] (0,1.0) .. controls ++(0,1.2) and ++(0,1.2) .. (2.25,1.0) to (2.25,-1);
    \node at (-.35,.75) {  $\lambda$};
    \node at (1.3,-.7) {\tiny $i$};
    \node at (2.05,-.7) {\tiny $j$};
    \node at (-.9,1.7) {\tiny $i$};
    \node at (-1.7,1.7) {\tiny $j$};
\end{tikzpicture}}
\\
& \qquad  =  \gamma
D_j D_{i}^{-1}
t_{ji}^{\prime -1}t_{ij}' t_{ij}^{-1}t_{ji}\;
\hackcenter{\begin{tikzpicture}[scale=0.7]
    \draw[thick,blue, ->] (0,0) .. controls (0,.5) and (.75,.5) .. (.75,1.0);
    \draw[thick, blue,->] (.75,0) .. controls (.75,.5) and (0,.5) .. (0,1.0);
    \draw[thick, blue] (0,0) .. controls ++(0,-.4) and ++(0,-.4) .. (-.75,0) to (-.75,2);
    \draw[thick, blue] (.75,0) .. controls ++(0,-1.2) and ++(0,-1.2) .. (-1.5,0) to (-1.55,2);
    \draw[thick, blue, ->] (.75,1.0) .. controls ++(0,.4) and ++(0,.4) .. (1.5,1.0) to (1.5,-1);
    \draw[thick, blue, ->] (0,1.0) .. controls ++(0,1.2) and ++(0,1.2) .. (2.25,1.0) to (2.25,-1);
    \node at (-.35,.75) {  $\lambda$};
    \node at (1.3,-.7) {\tiny $i$};
    \node at (2.05,-.7) {\tiny $j$};
    \node at (-.9,1.7) {\tiny $i$};
    \node at (-1.7,1.7) {\tiny $j$};
\end{tikzpicture}}
\;\; = \;\;  \gamma
\hackcenter{\begin{tikzpicture}[scale=0.7]
    \draw[thick,blue, ->] (0,0) .. controls (0,.5) and (.75,.5) .. (.75,1.0);
    \draw[thick, blue,->] (.75,0) .. controls (.75,.5) and (0,.5) .. (0,1.0);
    \draw[thick, blue] (0,0) .. controls ++(0,-.4) and ++(0,-.4) .. (-.75,0) to (-.75,2);
    \draw[thick, blue] (.75,0) .. controls ++(0,-1.2) and ++(0,-1.2) .. (-1.5,0) to (-1.55,2);
    \draw[thick, blue, ->] (.75,1.0) .. controls ++(0,.4) and ++(0,.4) .. (1.5,1.0) to (1.5,-1);
    \draw[thick, blue, ->] (0,1.0) .. controls ++(0,1.2) and ++(0,1.2) .. (2.25,1.0) to (2.25,-1);
    \node at (-.35,.75) {  $\lambda$};
    \node at (1.3,-.7) {\tiny $i$};
    \node at (2.05,-.7) {\tiny $j$};
    \node at (-.9,1.7) {\tiny $i$};
    \node at (-1.7,1.7) {\tiny $j$};
\end{tikzpicture}} \label{eq:gimel-cross}
\\
& \qquad = \gamma \;\;
 \hackcenter{\begin{tikzpicture}[xscale=-1.0, scale=0.7]
    \draw[thick, blue, ->] (0,0) .. controls (0,.5) and (.75,.5) .. (.75,1.0);
    \draw[thick, blue, ->] (.75,0) .. controls (.75,.5) and (0,.5) .. (0,1.0);
    \draw[thick, blue] (0,0) .. controls ++(0,-.4) and ++(0,-.4) .. (-.75,0) to (-.75,2);
    \draw[thick, blue] (.75,0) .. controls ++(0,-1.2) and ++(0,-1.2) .. (-1.5,0) to (-1.55,2);
    \draw[thick, blue, ->] (.75,1.0) .. controls ++(0,.4) and ++(0,.4) .. (1.5,1.0) to (1.5,-1);
    \draw[thick, blue, ->] (0,1.0) .. controls ++(0,1.2) and ++(0,1.2) .. (2.25,1.0) to (2.25,-1);
    \node at (1.2,.75) {  $\lambda$};
    \node at (1.3,-.7) {\tiny $j$};
    \node at (2.05,-.7) {\tiny $i$};
    \node at (-.9,1.7) {\tiny $j$};
    \node at (-1.7,1.7) {\tiny $i$};
\end{tikzpicture}}
\;\; = \;\;
\gimel \left( \; \;
 \hackcenter{\begin{tikzpicture}[xscale=-1.0, scale=0.7]
    \draw[thick,  ->] (0,0) .. controls (0,.5) and (.75,.5) .. (.75,1.0);
    \draw[thick,  ->] (.75,0) .. controls (.75,.5) and (0,.5) .. (0,1.0);
    \draw[thick] (0,0) .. controls ++(0,-.4) and ++(0,-.4) .. (-.75,0) to (-.75,2);
    \draw[thick] (.75,0) .. controls ++(0,-1.2) and ++(0,-1.2) .. (-1.5,0) to (-1.55,2);
    \draw[thick,  ->] (.75,1.0) .. controls ++(0,.4) and ++(0,.4) .. (1.5,1.0) to (1.5,-1);
    \draw[thick,  ->] (0,1.0) .. controls ++(0,1.2) and ++(0,1.2) .. (2.25,1.0) to (2.25,-1);
    \node at (1.2,.75) {  $\lambda$};
    \node at (1.3,-.7) {\tiny $j$};
    \node at (2.05,-.7) {\tiny $i$};
    \node at (-.9,1.7) {\tiny $j$};
    \node at (-1.7,1.7) {\tiny $i$};
\end{tikzpicture}} \; \right)\nn
\end{align}
Where in \eqref{eq:gimel-cross} we used properties of the coefficients $Q$ and $Q'$ from Definition~\ref{eq:Q} and \eqref{eq:PiPj} when $(\alpha_i,\alpha_j)=-1$.

Similarly, if $\gamma$ represents the rescaling of an $ij$-labelled crossing then
\begin{align}
& \gimel \left(\;\;
 \hackcenter{\begin{tikzpicture}[scale=0.8]
    \draw[thick,<-] (0,0) .. controls ++(0,.5) and ++(0,-.5) .. (.75,1);
    \draw[thick] (.75,0) .. controls ++(0,.5) and ++(0,-.5) .. (0,1);
    \draw[thick, ->] (0,1 ) .. controls ++(0,.5) and ++(0,-.5) .. (.75,2);
    \draw[thick] (.75,1) .. controls ++(0,.5) and ++(0,-.5) .. (0,2);
        \node at (-.2,.15) {\tiny $i$};
    \node at (.95,.15) {\tiny $j$};
\node at (1.1,1.25) { $\lambda$};
\end{tikzpicture}} \;\; \right)
\;\; = \;\;
\gimel \left( \;\;
\hackcenter{\begin{tikzpicture}[xscale=-1.0, scale=0.7]
    \draw[thick, ->] (0,0) .. controls (0,.5) and (.75,.5) .. (.75,1.0);
    \draw[thick, ->] (.75,-.5) to (.75,0) .. controls (.75,.5) and (0,.5) .. (0,1.0) to (0,1.5);
    \draw[thick] (0,0) .. controls ++(0,-.4) and ++(0,-.4) .. (-.75,0) to (-.75,1.5);
    %\draw[thick] (.75,0) .. controls ++(0,-1.2) and ++(0,-1.2) .. (-1.5,0) to (-1.55,2);
    \draw[thick, ->] (.75,1.0) .. controls ++(0,.4) and ++(0,.4) .. (1.5,1.0) to (1.5,-.5);
    %\draw[thick, ->] (0,1.0) .. controls ++(0,1.2) and ++(0,1.2) .. (2.25,1.0) to (2.25,-1);
    \node at (-1.15,.55) {  $\lambda$};
    \node at (1.75,-.2) {\tiny $i$};
    \node at (.55,-.2) {\tiny $j$};
    \node at (-.9,1.2) {\tiny $i$};
    \node at (.25,1.2) {\tiny $j$};
\node at (.56,2.5) {
    \begin{tikzpicture}[xscale=1, scale=0.7]
    \draw[thick, ->] (0,0) .. controls (0,.5) and (.75,.5) .. (.75,1.0);
    \draw[thick, ->] (.75,-.5) to (.75,0) .. controls (.75,.5) and (0,.5) .. (0,1.0) to (0,1.5);
    \draw[thick,<-] (0,0) .. controls ++(0,-.4) and ++(0,-.4) .. (-.75,0) to (-.75,1.5);
    %\draw[thick] (.75,0) .. controls ++(0,-1.2) and ++(0,-1.2) .. (-1.5,0) to (-1.55,2);
    \draw[thick] (.75,1.0) .. controls ++(0,.4) and ++(0,.4) .. (1.5,1.0) to (1.5,-.5);
%    \node at (1.75,-.2) {\tiny $i$};
%    \node at (1,-.2) {\tiny $j$};
    \node at (-.9,1.2) {\tiny $i$};
    \node at (.25,1.2) {\tiny $j$};
%\node at (1.1,1.25) { $\lambda$};
    \end{tikzpicture}
    };
\end{tikzpicture}}
\;\; \right)
\;\; = \;\;
%t_{ij}(t_{ij}')^{-1} D_i  \;
\gamma \cdot  D_i^{-\l_i+1} \frac{\sfc_{i,\l}^{ +}}{\sfc_{i,\l}^{\prime +}}
\cdot
% lambda = lambda - \alpha_i + \alpha_j  so \lambda_i = \lambda_i -3
D_{i}^{\l_i-1+  (\alpha_i,\alpha_j)} \frac{\sfc_{i,\l+\alpha_j}^{\prime +}}{\sfc_{i,\l+\alpha_j}^{+}}
 \hackcenter{\begin{tikzpicture}[scale=0.8]
    \draw[thick, blue, <-] (0,0) .. controls ++(0,.5) and ++(0,-.5) .. (.75,1);
    \draw[thick, blue] (.75,0) .. controls ++(0,.5) and ++(0,-.5) .. (0,1);
    \draw[thick,blue, ->] (0,1 ) .. controls ++(0,.5) and ++(0,-.5) .. (.75,2);
    \draw[thick, blue ] (.75,1) .. controls ++(0,.5) and ++(0,-.5) .. (0,2);
        \node at (-.2,.15) {\tiny $i$};
    \node at (.95,.15) {\tiny $j$};
\node at (1.1,1.25) { $\lambda$};
\end{tikzpicture}} \nn
\\
& \quad \;\; = \;\;
\gamma \cdot
D_{i}^{(\alpha_i,\alpha_j)}
\frac{\sfc_{i,\l+\alpha_j}^{\prime +}}{\sfc_{i,\l}^{\prime +}}
\cdot
\frac{\sfc_{i,\l}^{ +}}{\sfc_{i,\l+\alpha_j}^{+}}
 \hackcenter{\begin{tikzpicture}[scale=0.8]
    \draw[thick, blue, <-] (0,0) .. controls ++(0,.5) and ++(0,-.5) .. (.75,1);
    \draw[thick, blue] (.75,0) .. controls ++(0,.5) and ++(0,-.5) .. (0,1);
    \draw[thick,blue, ->] (0,1 ) .. controls ++(0,.5) and ++(0,-.5) .. (.75,2);
    \draw[thick, blue ] (.75,1) .. controls ++(0,.5) and ++(0,-.5) .. (0,2);
        \node at (-.2,.15) {\tiny $i$};
    \node at (.95,.15) {\tiny $j$};
\node at (1.1,1.25) { $\lambda$};
\end{tikzpicture}}
%\;\; = \;\;
%\gamma \cdot
%D_{i}^{(\alpha_i,\alpha_j)}
%\frac{\sfc_{i,\l+\alpha_j}^{\prime +}}{\sfc_{i,\l}^{\prime +}}
%\cdot
%\frac{\sfc_{i,\l}^{ +}}{\sfc_{i,\l+\alpha_j}^{+}}
% \hackcenter{\begin{tikzpicture}[scale=0.8]
%    \draw[thick, blue, <-] (0,0) .. controls ++(0,.5) and ++(0,-.5) .. (.75,1);
%    \draw[thick, blue] (.75,0) .. controls ++(0,.5) and ++(0,-.5) .. (0,1);
%    \draw[thick,blue, ->] (0,1 ) .. controls ++(0,.5) and ++(0,-.5) .. (.75,2);
%    \draw[thick, blue ] (.75,1) .. controls ++(0,.5) and ++(0,-.5) .. (0,2);
%        \node at (-.2,.15) {\tiny $i$};
%    \node at (.95,.15) {\tiny $j$};
%\node at (1.1,1.25) { $\lambda$};
%\end{tikzpicture}}
 \label{XX1}
 \;\; \refequal{\eqref{eq:c-change-sl}} \;\;
\gamma \cdot
D_{i}^{(\alpha_i,\alpha_j)}
t_{ij}' t_{ij}^{-1}
 \hackcenter{\begin{tikzpicture}[scale=0.8]
    \draw[thick, blue, <-] (0,0) .. controls ++(0,.5) and ++(0,-.5) .. (.75,1);
    \draw[thick, blue] (.75,0) .. controls ++(0,.5) and ++(0,-.5) .. (0,1);
    \draw[thick,blue, ->] (0,1 ) .. controls ++(0,.5) and ++(0,-.5) .. (.75,2);
    \draw[thick, blue ] (.75,1) .. controls ++(0,.5) and ++(0,-.5) .. (0,2);
        \node at (-.2,.15) {\tiny $i$};
    \node at (.95,.15) {\tiny $j$};
\node at (1.1,1.25) { $\lambda$};
\end{tikzpicture}}
%\; = \; (1+\delta_{ij} \frac{t_{ii}'}{t_{ii}})
% \hackcenter{\begin{tikzpicture}[scale=0.8]
%    \draw[thick, blue, <-] (0,0) .. controls ++(0,.5) and ++(0,-.5) .. (.75,1);
%    \draw[thick, blue] (.75,0) .. controls ++(0,.5) and ++(0,-.5) .. (0,1);
%    \draw[thick,blue, ->] (0,1 ) .. controls ++(0,.5) and ++(0,-.5) .. (.75,2);
%    \draw[thick, blue ] (.75,1) .. controls ++(0,.5) and ++(0,-.5) .. (0,2);
%        \node at (-.2,.15) {\tiny $i$};
%    \node at (.95,.15) {\tiny $j$};
%\node at (1.1,1.25) { $\lambda$};
%\end{tikzpicture}}
%\; = \;
% \hackcenter{\begin{tikzpicture}[scale=0.8]
%    \draw[thick, blue, <-] (0,0) .. controls ++(0,.5) and ++(0,-.5) .. (0,2);
%    \draw[thick,blue, ->] (.75,0 ) .. controls ++(0,.5) and ++(0,-.5) .. (.75,2);
%        \node at (-.2,.15) {\tiny $i$};
%    \node at (.95,.15) {\tiny $j$};
%\node at (1.1,1.25) { $\lambda$};
%\end{tikzpicture}}
%\; = \;
%\gimel \left( \;
% \hackcenter{\begin{tikzpicture}[scale=0.8]
%    \draw[thick,   <-] (0,0) .. controls ++(0,.5) and ++(0,-.5) .. (0,2);
%    \draw[thick,  ->] (.75,0 ) .. controls ++(0,.5) and ++(0,-.5) .. (.75,2);
%        \node at (-.2,.15) {\tiny $i$};
%    \node at (.95,.15) {\tiny $j$};
%\node at (1.1,1.25) { $\lambda$};
%\end{tikzpicture}} \; \right)
\end{align}
Using the specific value of $\gamma$ from \eqref{eq:gimel-KLR} that depends on the relationship of $i$ and $j$ in the $ij$-crossing, it follows that
\begin{equation} \label{XX2}
 \gamma \cdot
D_{i}^{(\alpha_i,\alpha_j)}
t_{ij}' t_{ij}^{-1}
=
\left\{
  \begin{array}{ll}
      t_{ii}'/ t_{ii} = \beta_{i}'/ \beta_{i}   & \hbox{if $i=j$, } \\
     1 & \hbox{otherwise,  }
  \end{array}
\right.
\end{equation}
where we have made use of \eqref{AA1} when $i <j$ and $i \cdot j =-1$.
This, and a similar computation with the other orientation, shows that the mixed relations are preserved.

For the $\mf{sl}_2$ relation we note that bubbles (real or fake) rescale as
\begin{align}
\gimel \left(\;\;
\hackcenter{ \begin{tikzpicture} [scale=.8]
    % Bubble
 \draw  (-.75,1) arc (360:180:.45cm) [thick];
 \draw[<-](-.75,1) arc (0:180:.45cm) [thick];
     \filldraw  [black] (-1.55,.75) circle (2.5pt);
        \node at (-1.3,.3) { $\scriptstyle \ast +j$};
        \node at (-1.4,1.7) { $i $};
 %% L
 \node at (-.2,1.5) { $\lambda $};
\end{tikzpicture}}
\;\; \right)
\;\; = \;\;
D_i^{j}   \frac{\sfc_{i,\l}^{ +}}{\sfc_{i,\l}^{\prime +}}
 \;\;
\hackcenter{ \begin{tikzpicture} [scale=.8, blue]
    % Bubble
 \draw  (-.75,1) arc (360:180:.45cm) [thick];
 \draw[<-](-.75,1) arc (0:180:.45cm) [thick];
     \filldraw  [black] (-1.55,.75) circle (2.5pt);
        \node at (-1.3,.3) { $\scriptstyle \ast + j$};
        \node at (-1.4,1.7) { $i $};
 %% L
 \node at (-.2,1.5) { $\lambda $};
\end{tikzpicture}}
\qquad \quad
\gimel \left(\;\;
\hackcenter{ \begin{tikzpicture} [scale=.8]
    % Bubble
 \draw  (-.75,1) arc (360:180:.45cm) [thick];
 \draw[->](-.75,1) arc (0:180:.45cm) [thick];
     \filldraw  [black] (-1.55,.75) circle (2.5pt);
        \node at (-1.3,.3) { $\scriptstyle \ast +j$};
        \node at (-1.4,1.7) { $i $};
 %% L
 \node at (-.2,1.5) { $\lambda $};
\end{tikzpicture}}
\;\; \right)
\;\; = \;\;
D_i^{j}   \frac{\sfc_{i,\l+\alpha_i}^{\prime +}}{\sfc_{i,\l+\alpha_i}^{+}}
 \;\;
\hackcenter{ \begin{tikzpicture} [scale=.8, blue]
    % Bubble
 \draw  (-.75,1) arc (360:180:.45cm) [thick];
 \draw[->](-.75,1) arc (0:180:.45cm) [thick];
     \filldraw  [black] (-1.55,.75) circle (2.5pt);
        \node at (-1.3,.3) { $\scriptstyle \ast +j$};
        \node at (-1.4,1.7) { $i $};
 %% L
 \node at (-.2,1.5) { $\lambda $};
\end{tikzpicture}}
\end{align}
Hence,
\begin{align}
& \gimel \left( \beta_{i,\l}
\sum_{\overset{f_1+f_2+f_3}{=\l_i-1}}\hackcenter{
 \begin{tikzpicture}[scale=0.8]
 \draw[thick,->] (0,-1.0) .. controls ++(0,.5) and ++ (0,.5) .. (.8,-1.0) node[pos=.75, shape=coordinate](DOT1){};
  \draw[thick,<-] (0,1.0) .. controls ++(0,-.5) and ++ (0,-.5) .. (.8,1.0) node[pos=.75, shape=coordinate](DOT3){};
 \draw[thick,->] (0,0) .. controls ++(0,-.45) and ++ (0,-.45) .. (.8,0)node[pos=.25, shape=coordinate](DOT2){};
 \draw[thick] (0,0) .. controls ++(0,.45) and ++ (0,.45) .. (.8,0);
 \draw (-.15,.7) node { $\scs i$};
\draw (1.05,0) node { $\scs i$};
\draw (-.15,-.7) node { $\scs i$};
% \draw (.3,.125) node {};
% \draw  (0,0) arc (180:360:0.5cm) [thick];
% \draw[,<-](1,0) arc (0:180:0.5cm) [thick];
%\filldraw  [black] (.1,-.25) circle (2.5pt);
 \node at (.95,.65) {\tiny $f_3$};
 \node at (-.55,-.05) {\tiny $\overset{-\l_i-1}{+f_2}$};
  \node at (.95,-.65) {\tiny $f_1$};
 \node at (1.5,.3) { $\lambda $};
 \filldraw[thick]  (DOT3) circle (2.5pt);
  \filldraw[thick]  (DOT2) circle (2.5pt);
  \filldraw[thick]  (DOT1) circle (2.5pt);
\end{tikzpicture} }  \; \right)
 \; = \;
\beta_{i,\l}
\sum_{\overset{f_1+f_2+f_3}{=\l_i-1}}
D_i^{f_1+f_3}
\cdot
 D_i^{f_2}   \frac{\sfc_{i,\l+\alpha_i}^{\prime +}} {\sfc_{i,\l+\alpha_i}^{+}}
\cdot
D_i^{-\l_i+1} \frac{\sfc_{i,\l}^{ +}}{\sfc_{i,\l}^{\prime +}}
\hackcenter{
 \begin{tikzpicture}[scale=0.8]
 \draw[thick,blue, ->] (0,-1.0) .. controls ++(0,.5) and ++ (0,.5) .. (.8,-1.0) node[pos=.75, shape=coordinate](DOT1){};
  \draw[thick,blue,<-] (0,1.0) .. controls ++(0,-.5) and ++ (0,-.5) .. (.8,1.0) node[pos=.75, shape=coordinate](DOT3){};
 \draw[thick,blue,->] (0,0) .. controls ++(0,-.45) and ++ (0,-.45) .. (.8,0)node[pos=.25, shape=coordinate](DOT2){};
 \draw[thick, blue] (0,0) .. controls ++(0,.45) and ++ (0,.45) .. (.8,0);
 \draw (-.15,.7) node { $\scs i$};
\draw (1.05,0) node { $\scs i$};
\draw (-.15,-.7) node { $\scs i$};
% \draw (.3,.125) node {};
% \draw  (0,0) arc (180:360:0.5cm) [thick];
% \draw[,<-](1,0) arc (0:180:0.5cm) [thick];
%\filldraw  [black] (.1,-.25) circle (2.5pt);
 \node at (.95,.65) {\tiny $f_3$};
 \node at (-.55,-.05) {\tiny $\overset{-\l_i-1}{+f_2}$};
  \node at (.95,-.65) {\tiny $f_1$};
 \node at (1.5,.3) { $\lambda $};
 \filldraw[thick, blue]  (DOT3) circle (2.5pt);
  \filldraw[thick, blue]  (DOT2) circle (2.5pt);
  \filldraw[thick, blue]  (DOT1) circle (2.5pt);
\end{tikzpicture} }
\nn
\\
& \quad  \; = \;
\beta_{i,\l}
\sum_{\overset{f_1+f_2+f_3}{=\l_i-1}}
 \frac{\sfc_{i,\l}^{ +}} {\sfc_{i,\l+\alpha_i}^{+}}
\cdot
 \frac{\sfc_{i,\l+\alpha_i}^{\prime +}}{\sfc_{i,\l}^{\prime +}}
\hackcenter{
 \begin{tikzpicture}[scale=0.8]
 \draw[thick,blue, ->] (0,-1.0) .. controls ++(0,.5) and ++ (0,.5) .. (.8,-1.0) node[pos=.75, shape=coordinate](DOT1){};
  \draw[thick,blue,<-] (0,1.0) .. controls ++(0,-.5) and ++ (0,-.5) .. (.8,1.0) node[pos=.75, shape=coordinate](DOT3){};
 \draw[thick,blue,->] (0,0) .. controls ++(0,-.45) and ++ (0,-.45) .. (.8,0)node[pos=.25, shape=coordinate](DOT2){};
 \draw[thick, blue] (0,0) .. controls ++(0,.45) and ++ (0,.45) .. (.8,0);
 \draw (-.15,.7) node { $\scs i$};
\draw (1.05,0) node { $\scs i$};
\draw (-.15,-.7) node { $\scs i$};
% \draw (.3,.125) node {};
% \draw  (0,0) arc (180:360:0.5cm) [thick];
% \draw[,<-](1,0) arc (0:180:0.5cm) [thick];
%\filldraw  [black] (.1,-.25) circle (2.5pt);
 \node at (.95,.65) {\tiny $f_3$};
 \node at (-.55,-.05) {\tiny $\overset{-\l_i-1}{+f_2}$};
  \node at (.95,-.65) {\tiny $f_1$};
 \node at (1.5,.3) { $\lambda $};
 \filldraw[thick, blue]  (DOT3) circle (2.5pt);
  \filldraw[thick, blue]  (DOT2) circle (2.5pt);
  \filldraw[thick, blue]  (DOT1) circle (2.5pt);
\end{tikzpicture} }
\; \refequal{\eqref{eq:c-change-sl}} \;
\beta_{i,\l}
\sum_{\overset{f_1+f_2+f_3}{=\l_i-1}}
 t_{ii}^{-1} t_{ii}'
\hackcenter{
 \begin{tikzpicture}[scale=0.8]
 \draw[thick,blue, ->] (0,-1.0) .. controls ++(0,.5) and ++ (0,.5) .. (.8,-1.0) node[pos=.75, shape=coordinate](DOT1){};
  \draw[thick,blue,<-] (0,1.0) .. controls ++(0,-.5) and ++ (0,-.5) .. (.8,1.0) node[pos=.75, shape=coordinate](DOT3){};
 \draw[thick,blue,->] (0,0) .. controls ++(0,-.45) and ++ (0,-.45) .. (.8,0)node[pos=.25, shape=coordinate](DOT2){};
 \draw[thick, blue] (0,0) .. controls ++(0,.45) and ++ (0,.45) .. (.8,0);
 \draw (-.15,.7) node { $\scs i$};
\draw (1.05,0) node { $\scs i$};
\draw (-.15,-.7) node { $\scs i$};
% \draw (.3,.125) node {};
% \draw  (0,0) arc (180:360:0.5cm) [thick];
% \draw[,<-](1,0) arc (0:180:0.5cm) [thick];
%\filldraw  [black] (.1,-.25) circle (2.5pt);
 \node at (.95,.65) {\tiny $f_3$};
 \node at (-.55,-.05) {\tiny $\overset{-\l_i-1}{+f_2}$};
  \node at (.95,-.65) {\tiny $f_1$};
 \node at (1.5,.3) { $\lambda $};
 \filldraw[thick, blue]  (DOT3) circle (2.5pt);
  \filldraw[thick, blue]  (DOT2) circle (2.5pt);
  \filldraw[thick, blue]  (DOT1) circle (2.5pt);
\end{tikzpicture} }
\nn
\\
& \quad  \; = \;
\beta'_{i,\l}
\sum_{\overset{f_1+f_2+f_3}{=\l_i-1}}
\hackcenter{
 \begin{tikzpicture}[scale=0.8]
 \draw[thick,blue, ->] (0,-1.0) .. controls ++(0,.5) and ++ (0,.5) .. (.8,-1.0) node[pos=.75, shape=coordinate](DOT1){};
  \draw[thick,blue,<-] (0,1.0) .. controls ++(0,-.5) and ++ (0,-.5) .. (.8,1.0) node[pos=.75, shape=coordinate](DOT3){};
 \draw[thick,blue,->] (0,0) .. controls ++(0,-.45) and ++ (0,-.45) .. (.8,0)node[pos=.25, shape=coordinate](DOT2){};
 \draw[thick, blue] (0,0) .. controls ++(0,.45) and ++ (0,.45) .. (.8,0);
 \draw (-.15,.7) node { $\scs i$};
\draw (1.05,0) node { $\scs i$};
\draw (-.15,-.7) node { $\scs i$};
% \draw (.3,.125) node {};
% \draw  (0,0) arc (180:360:0.5cm) [thick];
% \draw[,<-](1,0) arc (0:180:0.5cm) [thick];
%\filldraw  [black] (.1,-.25) circle (2.5pt);
 \node at (.95,.65) {\tiny $f_3$};
 \node at (-.55,-.05) {\tiny $\overset{-\l_i-1}{+f_2}$};
  \node at (.95,-.65) {\tiny $f_1$};
 \node at (1.5,.3) { $\lambda $};
 \filldraw[thick, blue]  (DOT3) circle (2.5pt);
  \filldraw[thick, blue]  (DOT2) circle (2.5pt);
  \filldraw[thick, blue]  (DOT1) circle (2.5pt);
\end{tikzpicture} }
\end{align}
This computation together with \eqref{XX1} and \eqref{XX2} prove this $\mf{sl}_2$-relation.  The other is proven similarly.
\end{proof}

 % ==============================================================================
% REFERENCES
%

%\bibliographystyle{plain}
%\bibliographystyle{amsalpha}
%\bibliography{bib_fincoljones}

%
%
% ==============================================================================

\end{document}